\pdfoutput=1
\RequirePackage{silence}
\documentclass[a4paper,11pt]{amsart}
\usepackage[hmarginratio={1:1},vmarginratio={1:1},lmargin=60.0pt,tmargin=60.0pt]{geometry}

\allowdisplaybreaks

\usepackage[numbers]{natbib}
\usepackage[utf8]{inputenc}

\usepackage{comment}

\usepackage{latexsym,exscale,mathtools,textcomp}
\usepackage{amssymb,amsmath,amsthm,amsfonts,mathrsfs,bbm,enumitem,stmaryrd}
\usepackage{blkarray}
\usepackage[table]{xcolor}
\usepackage{graphicx}
\usepackage{mathtools}
\usepackage{ytableau}
\usepackage{booktabs}

\usepackage{xparse}

\usepackage{dynkin-diagrams}

\setlist[enumerate]{itemsep=0.05cm,label=\emph{\upshape(\alph*)}}
\setlist[enumerate,2]{itemsep=0.05cm,label=\emph{\upshape(\roman*)}}

\usepackage{array}
\newcolumntype{C}{>{$}c<{$}}

\definecolor{mygray}{gray}{0.6}
\definecolor{mygraydark}{gray}{0.4}
\definecolor{mygraylight}{gray}{0.85}
\definecolor{spinach}{RGB}{46,139,87}
\definecolor{tomato}{RGB}{255,99,71}
\definecolor{orchid}{RGB}{143,40,194}
\definecolor{neon}{RGB}{77,77,255}
\definecolor{pumpkin}{RGB}{224,180,80}
\definecolor{citron}{RGB}{190,180,90}

\definecolor{lava}{RGB}{207,16,32}
\definecolor{cream}{RGB}{255,253,208}
\definecolor{verdigris}{RGB}{67,179,174}
\definecolor{Black}{RGB}{0,0,0}
\definecolor{mydarkblue}{RGB}{10,10,170}
\definecolor{darkspinach}{RGB}{20,70,20}
\definecolor{darktomato}{RGB}{155,40,30}
\definecolor{darkorchid}{RGB}{50,10,100}
\definecolor{darklava}{RGB}{150,8,16}

\usepackage{todonotes}

\let\emph\relax
\DeclareTextFontCommand{\emph}{\bfseries\em}

\newcommand{\placeholder}{{}_{-}}

\let\<=\langle
\let\>=\rangle

\renewcommand{\dots}{\text{...}}

\DeclarePairedDelimiterX{\set}[1]{\{}{\}}{\setargs{#1}}
\NewDocumentCommand{\setargs}{>{\SplitArgument{1}{|}}m}{\setargsaux#1}
\NewDocumentCommand{\setargsaux}{mm}
{\IfNoValueTF{#2}{#1} {#1\,\delimsize|\,\mathopen{}#2}}

\newcommand{\C}{\mathbb{C}}
\newcommand{\R}{\mathbb{R}}
\newcommand{\N}{\mathbb{Z}_{\geq 0}}

\newcommand{\Z}{\mathbb{Z}}

\renewcommand{\dim}[1][\C]{\mathrm{dim}_{#1}}

\newcommand{\ra}{\rightarrow}
\newcommand{\Atil}{\widetilde A_{2}}
\newcommand{\Aone}{\widetilde A_{1}}
\newcommand{\Hecke}{\mc H}

\newcommand{\wt}{\widetilde}

\newcommand{\nuop}{\nu}
\newcommand{\Rep}{\mathbf{Rep}\,}
\newcommand{\mc}{\mathcal}

\usepackage[all]{xy}
\usepackage{tikz}
\usetikzlibrary{cd}
\usetikzlibrary{decorations}
\usetikzlibrary{decorations.markings}
\usetikzlibrary{decorations.pathreplacing}
\usetikzlibrary{decorations.pathmorphing}
\usetikzlibrary{arrows.meta,shapes,positioning,matrix,calc}
\usetikzlibrary{shapes.callouts}
\usetikzlibrary{shapes.arrows}
\usetikzlibrary{tqft}
\usetikzlibrary{scopes,intersections}

\tikzset{
anchorbase/.style={baseline={([yshift=#1]current bounding box.center)}},
}

\let\oldlightning\lightning
\renewcommand{\lightning}{\textcolor{tomato}{\pmb{\oldlightning}}}

\usepackage{aliascnt,etoolbox}
\def\NewTheorem#1{%
\newaliascnt{#1}{equation}%
\newtheorem{#1}[#1]{#1}%
\aliascntresetthe{#1}%
\expandafter\def\csname #1autorefname\endcsname{#1}%
}
\def\equationautorefname~#1\null{(#1)\null}

\numberwithin{equation}{subsection}

\NewTheorem{Proposition}
\NewTheorem{Theorem}
\NewTheorem{Corollary}
\AtEndEnvironment{Corollary}{\null\hfill$\square$}%
\NewTheorem{Lemma}
\theoremstyle{definition}
\NewTheorem{Definition}
\AtEndEnvironment{Definition}{\null\hfill$\Diamond$}%
\NewTheorem{Notation}
\AtEndEnvironment{Notation}{\null\hfill$\Diamond$}%
\NewTheorem{Example}
\AtEndEnvironment{Example}{\null\hfill$\Diamond$}%
\NewTheorem{Question}
\AtEndEnvironment{Question}{\null\hfill$\Diamond$}%
\theoremstyle{remark}
\NewTheorem{Remark}
\AtEndEnvironment{Remark}{\null\hfill$\Diamond$}%

\usepackage[hypertexnames=false]{hyperref}
\usepackage{bookmark}
\hypersetup{
pdftoolbar=true,
pdfmenubar=true,
pdffitwindow=false,
pdfstartview={FitH},
pdftitle={Growth in affine Hecke categories},
pdfauthor={Kevin Coulembier, Jensen O'Sullivan, and Daniel Tubbenhauer},
pdfsubject={},
pdfcreator={Kevin Coulembier, Jensen O'Sullivan, and Daniel Tubbenhauer},
pdfproducer={Kevin Coulembier, Jensen O'Sullivan, and Daniel Tubbenhauer},
pdfkeywords={},
pdfnewwindow=true,
colorlinks=true,
linkcolor=mydarkblue,
citecolor=teal,
filecolor=magenta,
urlcolor=orchid,
linkbordercolor=lava,
citebordercolor=teal,
urlbordercolor=orchid,
linktocpage=true
}

\def\makeautorefname#1#2{\csdef{#1autorefname}{#2}}

\makeautorefname{section}{Section}%
\makeautorefname{subsection}{Section}%
\makeautorefname{subsubsection}{Section}%

\begin{document}
\title[Growth in affine Hecke categories]{Growth in affine Hecke categories}
\author[K. Coulembier, J. O'Sullivan, and D. Tubbenhauer]{Kevin Coulembier, Jensen O'Sullivan, and Daniel Tubbenhauer}

\address{K.C.: The University of Sydney, School of Mathematics and Statistics F07, Office Carslaw 717, NSW 2006, Australia, \href{https://www.maths.usyd.edu.au/u/kevinc/}{www.maths.usyd.edu.au/u/kevinc}, \href{https://orcid.org/0000-0003-0996-3965}{ORCID 0000-0003-0996-3965}}
\email{kevin.coulembier@sydney.edu.au}

\address{J.O.: The University of Sydney, School of Mathematics and Statistics F07, NSW 2006, Australia}
\email{J.OSullivan@maths.usyd.edu.au}

\address{D.T.: The University of Sydney, School of Mathematics and Statistics F07, Office Carslaw 827, NSW 2006, Australia, \href{http://www.dtubbenhauer.com}{www.dtubbenhauer.com}, \href{https://orcid.org/0000-0001-7265-5047}{ORCID 0000-0001-7265-5047}}
\email{daniel.tubbenhauer@sydney.edu.au}

\begin{abstract}
This paper studies the asymptotic growth of tensor powers in affine Hecke categories, or equivalently of powers of Kazhdan--Lusztig basis elements in affine Hecke algebras. We prove general bounds in arbitrary affine type, determine precise asymptotics in affine type A1, and establish corresponding results for several natural families in affine type A2.
\end{abstract}


\addtocontents{toc}{\protect\setcounter{tocdepth}{1}}

\maketitle

\tableofcontents

\section{Introduction}\label{sec:intro}

This paper studies a classical problem in a new asymptotic
setting: how tensor product decompositions grow in affine Hecke categories.

\subsection{Monoidal growth problems}\label{sec:gengrowth}

Tensor product decompositions are one of the oldest bookkeeping problems in
representation theory. Already for \(SL_2(\C)\), the
Clebsch--Gordan rule describes how a tensor product of two simple
representations splits into simples. In another classical direction,
the Littlewood--Richardson rule does the same job for polynomial
representations of general linear groups. These rules are not just
computational devices: they are among the basic ways in which representation
theory turns algebra into combinatorics.
The same problem also appears in physics. Coupling two angular momenta means
decomposing a tensor product of two \(SL_2(\C)\)-representations.
Thus the Clebsch--Gordan coefficients are, simultaneously, structure
constants in representation theory and the coefficients governing the
addition of quantum angular momenta.

From this point of view, tensor products are a classical story. The
asymptotic question is newer in spirit. Instead of decomposing one tensor
product, one fixes an object and asks what happens to its large tensor
powers, as pioneered, for example, in \cite{Bi-asymptotic-lie}. Thus the basic question becomes:
\[
\setlength{\fboxsep}{2pt}
\colorbox{blue!6}{
\begin{minipage}{0.85\textwidth}
\centering
\emph{Given an object \(X\) in a nice enough monoidal category, how many indecomposable
summands occur in \(X^{\otimes n}\), with multiplicity, as
\(n\to\infty\)?}
\end{minipage}}
\]
More concretely, if
\[
X^{\otimes n}\cong \bigoplus_i X_i^{\oplus m_i(n)},
\]
then we want to understand the growth of (the function/sequence)
\[
b_n:=\sum_i m_i(n).
\]
Let $\sim$ denote asymptotically equal. Following \cite{CoEtOsTu-growth-fractal}, in many natural examples:
\[
\setlength{\fboxsep}{2pt}
\colorbox{blue!6}{%
\begin{minipage}{0.85\textwidth}
\[
\text{``scalar-poly-exp''}
\colon
\quad
b_n\sim
\left(\begin{array}{c@{\quad\cdot\quad}c@{\quad\cdot\quad}c}
C & n^\tau & \beta^n \\
\text{\emph{scalar}} & \text{\emph{poly}} & \text{\emph{exp}}
\end{array}\right).
\]
\end{minipage}}
\]
The exponential term \(\beta^n\) is usually easy to determine 
(see e.g. \cite{CoOsTu-growth}): it is often a
dimension, or a Frobenius--Perron dimension. The polynomial correction
\(n^\tau\) is more structural. In many representation-theoretic examples,
\(\tau\) is dictated not by the particular object \(X\), but by the ambient
root system or category. The scalar \(C\) is usually the most delicate part. In some cases, the scalar-poly-exp form needs to be generalized to a slightly weaker expression
\begin{gather}\label{Eq:Theta}
C_1\cdot n^\tau\cdot \beta^n\leq b_n\leq C_2\cdot n^\tau\cdot \beta^n,
\end{gather}
for constants $C_1,C_2>0$, or simply $b_n\in\Theta(n^\tau\cdot\beta^n)$ in Bachmann--Landau notation.

\begin{Example}\label{E:Groups}
Let us work over the field of complex numbers.
Let \(G\) be a connected reductive complex algebraic group, and let \(V\)
be a finite dimensional faithful \(G\)-representation. Write
\[
V^{\otimes n}\cong \bigoplus_{\lambda\in X^+}
L(\lambda)^{\oplus m_\lambda(n)},\quad
b_n=\sum_{\lambda\in X^+}m_\lambda(n).
\]
By \cite{Bi-asymptotic-lie,CEO19}, tensor power asymptotics have the form
\[
b_n\sim C_V\cdot n^{-|\Phi^+|/2}\cdot (\dim V)^n .
\]
Thus the exponential term is the most naive possible one, namely
\((\dim V)^n\). The polynomial correction is root-system data ($|\Phi^+|$ is the number of positive roots):
\[
\tau=-|\Phi^+|/2.
\]
In particular, \(\tau\) is independent of the particular representation
\(V\), as long as \(V\) is large enough to see the whole group (i.e. faithful).

The scalar \(C_V\) is more subtle. It depends on \(V\), but it is still
governed by Lie theory. Indeed, one can express \(C_V\)
as a Gaussian integral over the real Cartan, and, when \(G\) is simple, use
the Macdonald--Mehta--Opdam identity to evaluate this integral in terms of
the root system and one representation dependent variance parameter.
In the rank-one case \(G=SL_2(\C)\), for the simple module
\(V=L(j)\) of dimension $j+1$, this specializes to
\[
b_n\sim
\frac{1}{\sqrt{\pi}}
\left(\frac{j(j+2)}{6}\right)^{-1/2}
\cdot
n^{-1/2}\cdot
(j+1)^n.
\]
(This can also easily proven by hand.)

The case of general groups is similar, which is the key motivation for the slogan ``scalar-poly-exp'' throughout.
\end{Example}

\begin{Remark}
Let us comment briefly on the context surrounding the ``scalar-poly-exp'' Ansatz.
\begin{enumerate}

\item In characteristic zero or finite categories or their analogs, the recent volume of literature is extensive. We refer the reader to \cite{GrTu-growth-diagram, He-growth-finite-groups, HeTu-growth-finite-monoids, HeTu-growth-diagram-monoids, LaPoReSo-growth-qgroup, LuZh-quiver-growth, OST-growth-quantum-groups} for a necessarily selective list of references, with additional citations appearing both above and throughout the main body of this text. Crucially, all of these works, with the sole exception of \cite{GrTu-growth-diagram}, confirm the ``scalar-poly-exp'' Ansatz. Because \cite{GrTu-growth-diagram} establishes the analog of this property for categories with superexponential growth, it is not a true counterexample to the underlying philosophy.

\item Over fields of positive characteristic, the situation is considerably more delicate. Nevertheless, the ``scalar-poly-exp'' property in the sense of \autoref{Eq:Theta} remains valid; see, for example, \cite{CoEtOsTu-growth-fractal, La-char2-story, La-linearly-reductive, She-trivial-summands, She-fractal-tilting-sl2}. We do not discuss positive characteristic effects in the present work.

\end{enumerate}
In line with this body of work, a main goal of this paper is to explore whether the Ansatz also holds in the Hecke categorical setting.
\end{Remark}

\subsection{Hecke categories}\label{sec:affinehecke-defn}

The present paper studies the general question from \autoref{sec:gengrowth} for Hecke categories. Since the relevant combinatorics of this category can be captured entirely by Kazhdan--Lusztig theory for the associated Hecke algebra, we will mostly just work within the latter. We focus mainly on the case of affine Weyl groups. Given the close link of these to representation theory, namely category \(\mathcal O\), tilting modules,
and group representations through the Satake isomorphism, they are a natural setting in which to push
asymptotic questions further.

Let \(W\) be a Coxeter group, and let
\[
\{b_w\mid w\in W\}
\]
be the Kazhdan--Lusztig basis of the Hecke algebra, specialized at \(v=1\).
For fixed \(w\in W\), write
\[
b_w^n=\sum_x a_x(n)b_x,
\]
where, crucially, $a_x(n)\in\N$,
and set
\[
\nu_n=\nu_n(b_w):=\nu(b_w^n):=\sum_x a_x(n).
\]
Categorically, this is the number of indecomposable summands, counted with
multiplicity, in the \(n\)th tensor power of the corresponding
indecomposable Soergel bimodule.

There is also a clear notion of ``dimension'' $\beta$ in the background. Namely, expand
\(b_w\) in the standard basis, sum the coefficients at \(v=1\), and call the
result \(\nu^\delta(b_w)\). Set
\[
\beta=\beta(w):=\nu^\delta(b_w).
\]
It is not difficult to see that $\limsup_{n\to\infty}\sqrt[n]{\nu_n(b_w)}\le\beta(w)$, see \autoref{lem:upper-bound}.
Thus \(\beta^n\) is the natural exponential scale of the problem:

\begin{Question}\label{q:first}
Is it always true that $\limsup_{n\to\infty}\sqrt[n]{\nu_n(b_w)}\le\beta(w)$ is an equality?
\end{Question}

If this is the case, the subsequent main
question concerns the polynomial loss between this easy dimension count and the
actual number of Kazhdan--Lusztig summands.

For $W$ a finite Coxeter group, the answer to \autoref{q:first} is affirmative, and we have a much stronger result:

\begin{Example}\label{ex:finite}
For finite Coxeter groups, the asymptotic $\nu_n\sim \tfrac{1}{|W|}\cdot n^0 \cdot\beta^n$ easily follows from Perron--Frobenius theory, see for instance \cite{LaTuVa-growth-pfdim}.
\end{Example}

The growth problem for arbitrary Coxeter groups seems intractable for the moment. A natural and interesting class of Coxeter groups beyond the finite case is formed by affine Weyl groups. For this case there is at least a natural guess for the subexponential factor.
Based on the results in the current paper we are led to the following concrete question. (Note that any proper parabolic subgroup of an affine Weyl group is a finite Coxeter group, so that the excluded case in \autoref{q:general} simply reduces to \autoref{ex:finite}.)

\begin{Question}\label{q:general}
Let $W$ be an affine Weyl group, with set $\Phi^+$ of positive roots of the associated finite Lie algebra. Is it true that for any $w\in W$ (that is not contained in a proper parabolic subgroup)
\[\nu_n\;\sim\; C_w\cdot n^{-|\Phi^+|/2}\cdot \beta^n,\]
for a constant $C_w$, or failing that
\[ C_w\cdot n^{-|\Phi^+|/2}\cdot \beta^n\;\le\; \nu_n\;\le\; D_w\cdot n^{-|\Phi^+|/2}\cdot \beta^n,\]
for constants $C_w,D_w>0$?
\end{Question}

Let us start with an example (with details in \autoref{sec:cox}):

\begin{Example}\label{E:First}
Our guiding example is the Coxeter element
\[
w=s_1s_2s_3
\]
in affine type \(\widetilde A_2\) (with Coxeter generators $s_1,s_2,s_3$). Here $\beta=8$ and $|\Phi^+|=3$. Computations give
\[
1,\,1,\,3,\,17,\,83,\,472,\,2813,\,17665,\ldots
\]
for the sequence \(\nu_n\). Normalizing this sequence to $a_n=\nu_n/(n^{-3/2}\cdot 8^n)$ gives 
\[
|a_{n+2}-a_{n+1}|<|a_{n+1}-a_n|<0.0001\text{ for }n\in\{20,\dots,30\},
\]
which suggests the scalar-poly-exp form from \autoref{q:general}
\[
\nu_n\sim C\cdot n^{-3/2}\cdot 8^n.
\]
The scalar \(C\) is harder to pin down.
\end{Example}

More generally, our first result (\autoref{sec:generalasymptotics}) says
that, in affine type, the answer is always close to the expected exponential
scale. If \(W\) is irreducible affine, with finite positive roots \(\Phi^+\),
then for every fixed \(w\in W\) in the projective cell (the lowest in Shi's notation \cite{Shi-two-sided-cell-I,Shi-two-sided-cell-II}) there is a constant \(C=C_w>0\) such that
\[
C\cdot n^{-|\Phi^+|}\cdot \beta^n
\leq
\nu_n
\leq
\beta^n
\]
for $n \geq 1$ (\autoref{thm:general-lower-upper}). Thus the exponential part is \(\beta^n\), giving an affirmative answer to~\autoref{q:general} for this case,
and the possible loss is at worst polynomial, with exponent \(-|\Phi^+|\).
The proof of the lower bound uses the large-scale geometry of affine Weyl
groups together with Weyl's dimension formula.

These general bounds are not always sharp. In the best cases the
Hecke categorical problem reduces to an ordinary tensor power problem for
representations of groups, and the method in \autoref{E:Groups} applies. Then the sharper
scalar-poly-exp form is precisely the one in (the stronger version of) \autoref{q:general}.
We also study a slightly larger class of elements which are not quite in
this representation theoretic situation, but become so after applying the
spherical projection. In affine type \(\widetilde A_2\), one of the main
points of the paper is that the same sharper polynomial exponent still
governs these elements, although we can only prove the weaker form in \autoref{q:general}. 

We first treat affine type \(\widetilde A_1\), the \(\mathfrak{sl}_2\) case in \autoref{sec:sl2}. This is
the rank one test case. Here everything can be computed explicitly, and for
every nonfinite element one gets the expression from \autoref{q:general}.

The main concrete case is affine type \(\widetilde A_2\). Here the geometry
can still be drawn. The alcoves fall into three visible regions: the center,
the wall, and the large region beyond the wall (which we call the projective cell). Using the explicit
combinatorics of Libedinsky--Patimo \cite{LiPa}, we show that for elements
starting on the wall, the wall itself eventually contributes only a
negligible part of the answer. The asymptotic growth is controlled by the
projective cell.

\begin{Remark}
That growth is governed by the projective cell mirrors a remarkable result in \cite{BrKo-tensor-group}, whereby a slight extension of the argument implies that eventually almost all summands of a tensor power of a faithful representation of a finite group are projective. Consequently, the growth problem in the finite group case is also controlled by projective modules.
\end{Remark}

For the most representation theoretic elements in affine type
\(\widetilde A_2\), the two-sided longest elements, in \autoref{sec:sl2} we get
\[
\nu(b_w^n)\sim C_w\cdot n^{-3/2}\cdot \beta^n.
\]
For only one-sided longest elements, a uniform beyond-the-wall estimate gives the
matching lower bound, and hence (in Bachmann--Landau notation)
\[
\nu(b_w^n)\in \Theta(n^{-3/2}\cdot \beta^n).
\]
Thus affine types \(\widetilde A_1\) and \(\widetilde A_2\) are the first nontrivial affine cases in
which the expected scalar-poly-exp form can be proved beyond the directly
representation theoretic setting.

The paper is organized as follows. In \autoref{sec:prelim} we recall
the affine Hecke algebra, and the spherical Satake formalism. In \autoref{sec:general} we prove the
general bounds in arbitrary affine type, and record the consequences of
Satake for one-sided and two-sided longest elements. In \autoref{sec:sl2}
we treat affine type \(\widetilde A_1\) explicitly. In \autoref{sec:sl3}
we study affine type \(\widetilde A_2\). In \autoref{S:C2} we discuss affine type \(\widetilde C_2\) experimentally using computer calculations. Finally,
\autoref{S:Details} contains some nasty beyond-the-wall computations used in type
\(\widetilde A_2\).
\medskip

\noindent\textbf{Acknowledgments.}
This paper is part of the second author’s PhD thesis, for which he receives untaxed, yet traceable money from the University of Sydney, in the form of the RTP Stipend Scholarship. Furthermore, this paper is part of the ARC Discovery Project DP250100762. KC was supported by ARC Future Fellowship FT220100125. DT acknowledges support from ARC Future Fellowship FT230100489, and notes that all constants are positive, but this should not be mistaken for optimism.

\section{Preliminaries}\label{sec:prelim}

We collect a few basic facts that we will use throughout. References for the
material below include, in order, \cite{Hu-coxeter},
\cite{EMTW-soergel-intro}, and \cite{Kn-spherical-KL}.

\subsection{Affine Weyl groups}\label{sec:Affineintro}

Let $(W,S)$ be an affine Weyl group (an irreducible Coxeter
system of affine Weyl type), with simple reflection generators
\[
S=\{s_0,s_1,\dots,s_r\},
\]
where $s_0$ is the affine simple reflection. The underlying Coxeter diagram of
an affine Weyl group determines the corresponding affine Dynkin diagram.
Removing the affine node then gives a finite Dynkin diagram. We fix the
corresponding complex semisimple adjoint group $G$. For example, if $W=\tilde{A}_n$, then $G=PGL_{n+1}$.  We choose conventions so
that the translation lattice of $W$ identifies with the root lattice of $G$. 

Let
$Q=\bigoplus_{i=1}^r \mathbb Z\alpha_i$
be this root lattice, and put \(Q_{\mathbb R}=\mathbb R\otimes_{\mathbb Z} Q\).
Let \(\Phi\subset Q_{\mathbb R}\) be the associated root system, and choose
positive roots \(\Phi^+\) with simple roots
$\Delta=\{\alpha_1,\ldots,\alpha_r\}$.
Write \(\alpha^\vee\) for the coroot associated to \(\alpha\in\Phi\), and use
the usual root-coroot pairing
$\langle\placeholder,\placeholder\rangle:Q_{\mathbb R}\times \Phi^\vee\to \mathbb R$. Let \(\omega_1,\ldots,\omega_r\) denote the fundamental weights, so that
$\langle \omega_i,\alpha_j^\vee\rangle=\delta_{ij}$.
The finite Weyl group \(W_f=\langle s_1,\ldots,s_r\rangle\) acts on
\(Q_{\mathbb R}\) by
$s_i(\lambda)=\lambda-\langle \lambda,\alpha_i^\vee\rangle\alpha_i$.
We write
\[
Q^+ = \{\lambda\in Q \mid \langle \lambda,\alpha_i^\vee\rangle\geq 0
\text{ for all } 1\leq i\leq r\}.
\]
With these conventions, the affine Weyl group satisfies
\[
W\cong Q\rtimes W_f.
\]
For $\lambda\in Q$, we write
\[
t_\lambda\in W
\]
for the corresponding translation. Thus elements of $W$ may be regarded as
finite Weyl group elements together with a translation parameter
$\lambda\in Q$.

Let $\ell\colon W\ra \N$
be the standard length function, and let $\leq$ denote the Bruhat order on $W$.

We will also use the usual alcove model for $W$; see, for example,
\cite[\S 4.3]{Hu-coxeter}. The affine reflection hyperplanes cut
$Q_\R$ into connected components, called alcoves. The fundamental alcove is
the alcove corresponding to the identity element, and $W$ acts simply
transitively on the set of alcoves. We will therefore freely identify an
element $w\in W$ with the alcove obtained from the fundamental alcove by the
action of $w$.

Crossing a wall of an alcove corresponds to multiplying on the right by a
simple reflection. Thus reduced expressions for $w$ correspond to minimal
galleries from the fundamental alcove to the alcove $w$, and $\ell(w)$ is the
length of such a gallery. Later, when considering rays in $Q_\R$, we will
record the sequence of alcoves crossed by the ray in this sense.

\begin{Example}
For visualizations of the alcove picture, see, for example,
\autoref{fig:a2-alcoves}.
\end{Example}

\subsection{The Hecke algebra}\label{secke:Hecke}
Let $W$ be an arbitrary Coxeter group, and $\Hecke_v(W)$ be the Hecke algebra of $W$ with generic parameter $v$, so that $\Hecke_v(W)$ is a free $\Z[v,v^{-1}]$-module.
Since the asymptotic arguments in this paper take place after specializing
$v=1$, we will work, unless explicitly stated otherwise, with the
specialized $\R$-algebra via $\Z[v,v^{-1}]\to\R,v\mapsto 1$
\[
\Hecke:=\R\otimes_{\Z[v,v^{-1}]}\Hecke_v(W)\cong \R [W].
\]
It has standard basis
\[
\{\delta_w\mid w\in W\},
\]
with multiplication
\[
\delta_x\delta_y=\delta_{xy}.
\]

We also use the specialization at $v=1$ of the Kazhdan--Lusztig basis \cite{KaLu-coxeter-hecke},
denoted
\[
\{b_w\mid w\in W\}.
\]
For a simple reflection $s\in S$ one has
\[
b_s=\delta_s+1.
\]
Every KL basis element expands in the standard basis as
\[
b_w=\sum_{x\leq w} a_{x,w}\delta_x,
\quad
a_{x,w}\in\N,
\quad
a_{w,w}=1.
\]
Moreover, we have
\[
b_xb_y=\sum_{w\in W} p_{x,y}^w b_w,\quad p_{x,y}^w\in\N.
\]

\begin{Definition}
Define
\[
\nu\colon \Hecke\ra \R
\]
by
\[
\nu\left(\sum_{w\in W} c_w b_w\right)
:=
\sum_{w\in W} c_w.
\]
We think of $\nu$ as the number of indecomposable summands.
\end{Definition}

Note that this map is not multiplicative: in general,
\[
\nu(b_x)\nu(b_y)=1,
\]
whereas
\[
\nu(b_xb_y)\neq 1,
\]
for almost every pair of $x,y \in W$.

\begin{Definition}
An element of $\Hecke$ is called KL-positive if it is an
$\R_{\geq 0}$-linear combination of KL basis elements. Similarly, it is called
standard-positive if it is an $\R_{\geq 0}$-linear combination of
standard basis elements.
\end{Definition}

\begin{Definition}
Define
\[
\nu^\delta\colon \Hecke\ra \R
\]
by
\[
\nu^\delta\left(\sum_{x\in W} c_x\delta_x\right)
:=
\sum_{x\in W} c_x.
\]
We think of $\nu^\delta$ as the dimension.
\end{Definition}

\begin{Remark}
 Through
categorification, the map $\nu^\delta$ of a KL positive element can be interpreted as the rank of
the corresponding Soergel bimodule. 
\end{Remark}

\begin{Lemma}\label{nu_del ringhom}
The map
\[
\nu^\delta\colon \Hecke\ra \R
\]
is an $\R$-algebra homomorphism.
\end{Lemma}

\begin{proof}
Immediate.
\end{proof}

We will also use the standard cell structure coming from the KL basis; see
\cite{KaLu-coxeter-hecke}. The left, right, and two-sided KL preorders on $W$
give rise to left, right, and two-sided cells. We will only use this structure
in a coarse way.

For an affine Weyl group, the two-sided cell of central importance below is
the lowest two-sided cell in the convention of \cite{Shi-two-sided-cell-I,Shi-two-sided-cell-II}. We call it the projective cell. Its geometry
admits a concrete alcove description, due to Shi
\cite{Shi-two-sided-cell-I,Shi-two-sided-cell-II}, which will be recalled
when needed in \autoref{sec:poly}. We will also use the standard description
of this cell in terms of finite left and right corrections around spherical
elements, as in
\cite{LuXi-canonical-left-cells,Xi-lowest-based-ring-II}.

\begin{Example}
As above, for visualizations see, for example,
\autoref{fig:a2-alcoves}.
\end{Example}

\subsection{The Satake isomorphism}\label{sec:satake}

We retain the adjoint group $G$ fixed in \autoref{sec:Affineintro}. Let
$\Rep G$ be the symmetric monoidal category of finite dimensional complex
representations of $G$, and let
\[
K_0(\Rep G)
\]
be its Grothendieck ring, with basis given by the isomorphism classes of finite
dimensional simple $G$-modules, written $[L(\lambda)]$.

Let $w_0\in W_f$ be the longest element. Consider the spherical idempotent
\[
e
=
\frac{1}{|W_f|}
\sum_{x\in W_f}\delta_x
=
\frac{b_{w_0}}{|W_f|}.
\]
Then
\[
e^2=e,
\]
and we define the spherical Hecke algebra (specialised at $v=1$) by
\[
\Hecke_{\mathrm{sph}}:=e\Hecke e.
\]

We first record which KL basis elements are compatible with the spherical
idempotent.

\begin{Lemma}\label{lem:longestrepiff}
A KL basis element $b_w$ satisfies
\[
b_we=b_w
\]
if and only if $w$ is the longest representative of its coset in $W/W_f$.
\end{Lemma}

\begin{proof}
Suppose first that \(b_we=b_w\). For every simple reflection \(s_i\in W_f\),
we have \(eb_{s_i}=2e\), and hence
\[
b_wb_{s_i}=b_web_{s_i}=2b_we=2b_w.
\]
If \(ws_i>w\), then \(b_{ws_i}\) occurs in \(b_wb_{s_i}\), which is impossible.
Thus \(ws_i<w\) for every simple reflection \(s_i\in W_f\). By standard
Coxeter theory, this is equivalent to \(w\) being the longest representative
of its coset in \(W/W_f\).

Conversely, suppose that \(w\) is the longest representative of its coset in
\(W/W_f\). Then by construction $b_w\in \Hecke b_{s_i}$ for each quasi-idempotent $b_{s_i}$, $i>0$, so
\(b_wb_{s_i}=2b_w\), or in other words
\(b_w\delta_{s_i}=b_w\) for every simple reflection \(s_i\in W_f\). Therefore
\(b_w\delta_x=b_w\) for every \(x\in W_f\), and so
\[
b_we
=
\frac{1}{|W_f|}
\sum_{x\in W_f} b_w\delta_x
=
b_w,
\]
as desired.
\end{proof}

\begin{Lemma}\label{L:TwoSided}
A KL basis element $b_w$ satisfies
\[
b_w\in \Hecke_{\mathrm{sph}}
\]
if and only if $w$ is the longest representative of its double coset in
\[
W_f\backslash W/W_f.
\]
\end{Lemma}

\begin{proof}
Immediate from \autoref{lem:longestrepiff}, applied on the left and on the
right.
\end{proof}

For $\lambda\in Q^+$, let $m_\lambda$
denote the longest representative of the double coset
\[
W_f t_\lambda W_f.
\]
By \autoref{L:TwoSided},
\[
b_{m_\lambda}\in \Hecke_{\mathrm{sph}}.
\]

The following is the form of the Satake isomorphism that we use.

\begin{Lemma}\label{lem:satake}
The combinatorial Satake isomorphism is an $\R$-algebra isomorphism
\[
\Hecke_{\mathrm{sph}}
\ra
\R\otimes_{\Z} K_0(\Rep G),
\]
which sends
\[
b_{m_\lambda}
\longmapsto
|W_f|[L(\lambda)]
=
[L(\lambda)^{\oplus |W_f|}]
\]
for every dominant $\lambda\in Q^+$. Under this isomorphism,
$\nu^\delta$ corresponds to the ordinary dimension map. In particular,
\[
\nu^\delta(b_{m_\lambda})
=
|W_f|\dim L(\lambda).
\]
\end{Lemma}

\begin{proof}
A classical result; details can, for example, be found in
\cite{Kn-spherical-KL}.
\end{proof}

In particular,
\[
\nu^\delta(e)=1,
\]
which corresponds under Satake to the fact that $e$ maps to the trivial
representation.

We will also need a one-sided version of the spherical construction.

\begin{Lemma}\label{lem:satake-one-sided}
Assume that
\[
b_we=b_w.
\]
Then
\[
eb_w\in \Hecke_{\mathrm{sph}}.
\]
Let
\[
V_w\in \R\otimes_{\Z}K_0(\Rep G)
\]
denote the image of $eb_w$ under the Satake isomorphism. Then
\[
(eb_w)^n=eb_w^n
\]
for all $n\geq 1$, and
\[
\dim V_w=\nu^\delta(b_w).
\]
\end{Lemma}

\begin{proof}
Since \(b_we=b_w\), we have
\[
eb_w=eb_we\in e\Hecke e=\Hecke_{\mathrm{sph}}.
\]
Moreover,
\[
(eb_w)^2=eb_web_w=e(b_we)b_w=eb_w^2,
\]
and induction gives \((eb_w)^n=eb_w^n\) for all \(n\geq 1\).

Finally, using \autoref{nu_del ringhom},
\[
\nu^\delta(eb_w)
=
\nu^\delta(e)\nu^\delta(b_w)
=
\nu^\delta(b_w).
\]
Since \(\nu^\delta\) corresponds under Satake to the dimension map, we obtain
$\dim V_w
=
\nu^\delta(eb_w)
=
\nu^\delta(b_w)$.
\end{proof}

\begin{Remark}
\autoref{lem:satake-one-sided} applies equally to elements $b_w$ satisfying
$eb_w=b_w$.
These are precisely the elements for which $w$ is the longest representative
of its coset in
$W_f\backslash W$. Throughout, we will either use $eb_w=b_w$ or $b_we=b_w$ to mean one-sided longest elements; by symmetry, our arguments work for both of them, and we always only state one.
\end{Remark}

\begin{Example}
Here and throughout, let us write, $PGLn$ instead of
$PGL_n(\C)$.
 For $\widetilde{A}_1$, the adjoint group on
the Satake side is $PGL2$, and one has
\[
b_{(12)^k1}
\longmapsto
2[L(2k)].
\]
For instance,
\[
b_{121}
\longmapsto
2[L(2)],
\]
where $L(2)$ is the three-dimensional representation of $PGL2$.

For $\widetilde{A}_2$, the adjoint group on the Satake side is $PGL3$, and one has, for
example,
\[
b_{1213121}
\longmapsto
6[L(\alpha_1+\alpha_2)],
\]
where $L(\alpha_1+\alpha_2)$ is the eight-dimensional representation of
$PGL3$.
\end{Example}

\section{General results}\label{sec:general}

In this section, we derive coarse general bounds for both
$\nu^\delta(b_w)$, the number of standard summands in a KL basis element,
and $\nu(b_w^n)$, the number of KL summands in the decomposition of a power
of a KL basis element.

\subsection{Setup}

Throughout this section, we let $(W,S)$ be an irreducible Coxeter
system of affine Weyl type.

\begin{Definition}
Fix $w\in W$. Define
\[
\nu_n(b_w):=\nu(b_w^n)\in\mathbb{Z}_{\ge 1},
\quad n\geq 0.
\]
(This is the number of indecomposable summands in the corresponding power of
the indecomposable Soergel bimodule. Moreover,
$\nu^\delta(b_w^n)$ corresponds to the length of the standard filtration.) Finally, let
\[
\beta=\beta(w):=\nu^\delta(b_w).
\]
Thus $\beta$ is the sum of the standard basis coefficients of $b_w$ at $v=1$.
\end{Definition}

We apply all of these notions to standard-positive or KL-positive elements,
i.e. finite $\R_{\geq 0}$-linear combinations of the respective basis
elements.

To fix conventions, we record the asymptotic notation used below.

\begin{Notation}
Given two (eventually nonzero) real-valued functions $f$, $g$ which are both defined on $\N$ or
some infinite subset, we will write
\begin{gather*}
\begin{aligned}
f\sim g
&\;\Leftrightarrow\;
\forall\varepsilon>0,\,\exists n_{0}
\;\text{ such that }\;
\fbox{$|\tfrac{f(n)}{g(n)}-1|<\varepsilon$},
\;\forall n>n_{0}
,
\\
f\in o(g)
&\;\Leftrightarrow\;
\forall\varepsilon>0,\,\exists n_{0}
\;\text{ such that }\;
\fbox{\(|f(n)|\leq \varepsilon |g(n)|\)},
\;\forall n>n_{0}
,
\\
f\in O(g)
&\;\Leftrightarrow\;
\exists C>0,\,\exists n_{0}
\;\text{ such that }\;
\fbox{$|f(n)|\leq C\cdot |g(n)|$},
\;\forall n>n_{0}
,
\\
f\in\Theta(g)
&\;\Leftrightarrow\;
\exists C_{1},C_{2}>0,\,\exists n_{0}
\;\text{ such that }\;
\fbox{$C_{1}\cdot g(n)\leq f(n)\leq C_{2}\cdot g(n)$},
\;\forall n>n_{0}
.
\\
\end{aligned}
\end{gather*}
We also often regard sequences $\nu_n=(\nu_0,\nu_1,\dots)$ as functions,
using the same notation.
\end{Notation}

\subsection{Growth results for \texorpdfstring{$\nu^\delta$}{nudelta}}\label{sec:poly}

We record a growth statement for KL basis elements along a ray. 

Let \(r(t)=tv + u\) be a ray (a half-line) in $Q_{\R}$ starting at $u$, where \(t\geq 0\) and \(v\not=0\) is a vector in $Q_{\R}$. We call $r$ \textit{affine} if $u \neq 0$.

We call \(r\) regular if 
\[
\langle v,\alpha^\vee\rangle\neq 0
\quad
\text{for all }\alpha\in \Phi^+.
\]

We say $r$ is nonregular if 
\[\langle v,\alpha^\vee\rangle= 0
\quad
\text{for some }\alpha\in \Phi^+.\]

We will deal with regular rays in the following discussion; see
\autoref{lem:singular rays} for the nonregular case.

Suppose $r$ is regular. Then, it will intersect the interiors of infinitely many alcoves in $Q_\R$. Record the words $w \in W$ for which the ray $r(t)$ hits the interior of the alcove corresponding to $w$. Let $l$ be the length of such a $w$, and write $r_l$ for the corresponding alcove. Thus, if the ray intersects the fundamental alcove, \(r_0=id\). The variable \(l\) ranges over the infinite subset of lengths occurring
along the ray. We restrict ourselves to the case in which the sequence $l$ of word lengths along the ray is strictly increasing. Refer to \autoref{rem:wrong-way} for the other case.

\begin{Remark}
Note that a regular ray can skip word lengths. For example, if we let \(1,2,3\) be the
generators of \(\widetilde{A_2}\), and \(\omega_1,\omega_2\) be the fundamental
weights, the ray defined by \(v=\omega_1+\omega_2\) hits alcoves
\[
id \ra 121 \ra 1213 \ra 1213121 \ra 12131213 \ra \dots\,.
\]
For this ray $r(t) = tv$, we therefore have the list of alcoves $r_0, r_3, r_4, r_7, r_8, \dots$.

See \autoref{sec:sl3} for notational details.
\end{Remark}

\begin{Remark}\label{rem:wrong-way}
Depending on $u$ and $v$, the sequence of word lengths $r_l$ may decrease for some finite time (intuitively, if the ray is moving ``towards the origin" from $u$). This means that, for some values of $l$, there are potentially two alcoves $r_l, r'_l$. However, after a finite point, the word lengths are strictly increasing along the ray, so this doesn't cause any complications in the asymptotic behaviour. We simply choose not to include this case for ease of proofs, however note the results in this section are unchanged in this context.
\end{Remark}

We set
\[
p_l=p_l(r):=\nu^\delta(b_{r_l}).
\]

\begin{Theorem}\label{thm:poly-projective-ray}
Let \(r(t)=tv\) be a regular ray starting at the origin. Then
\[
p_l(r)\in \Theta\bigl(l^{|\Phi^+|}\bigr),
\]
as \(l\to\infty\) through the lengths met by the ray.
\end{Theorem}

\begin{proof}
This will be proved in the next few lemmas. We use Shi's alcove description
\cite{Shi-two-sided-cell-I,Shi-two-sided-cell-II} of the projective two-sided
cell, which is the lowest cell in Shi's terminology. This description of the projective two-sided cell says that the alcoves indexed by its elements are precisely those lying outside all Shi strips, which, in formulas, is
\[
\left\{
x\in Q_{\R}\ \middle|\ 
\text{for every }\alpha\in \Phi^+,\ 
\langle x,\alpha^\vee\rangle\geq 0
\text{ or }
\langle x,\alpha^\vee\rangle\leq -1
\right\}.
\]
Thus the projective cell is approximately a union of chambers for the finite Weyl group.

\begin{Lemma}\label{lem:regular-rays-projective}
Let \(r(t)=tv\) be a regular ray. Then \(r\) is eventually contained in the
projective cell. Moreover, once \(r\) is in the projective
cell, it never leaves it.
\end{Lemma}

\begin{proof}
Fix \(\alpha\in \Phi^+\). Since \(r\) is regular,
\[
\langle v,\alpha^\vee\rangle\neq 0.
\]
Thus \(\langle tv,\alpha^\vee\rangle\) tends either to \(+\infty\) or to
\(-\infty\). Hence, for all sufficiently large \(t\), we have either
\[
\langle tv,\alpha^\vee\rangle\geq 0
\quad
\text{or}
\quad
\langle tv,\alpha^\vee\rangle\leq -1.
\]
Since \(\Phi^+\) is finite, this holds simultaneously for all
\(\alpha\in\Phi^+\) after increasing \(t\). By the description recalled above, these inequalities place the ray in the projective cell. The same
inequalities also show that the ray cannot re-enter any
forbidden strip
\[
\{x\mid -1<\langle x,\alpha^\vee\rangle<0\}.
\]
Thus it remains in the projective cell.
\end{proof}

By \autoref{lem:regular-rays-projective}, the ray is eventually contained in
the projective cell. In the projective cell, every element is at uniformly
bounded Coxeter distance from a Satake element \(m_\mu\in W\), with
\(b_{m_\mu}\in\Hecke_{\mathrm{sph}}\), and with constants depending only on
the affine type. This follows from the following well-known lemma.

\begin{Lemma}\label{lem:projective-close-to-satake}
There is a finite set \(F\subset W\), depending only on the affine type, with
the following property. If \(w \in W\) lies in the projective cell, then there
are \(x,y\in F\) and a dominant weight \(\mu\) such that
\begin{align*}
w=x\,m_\mu y,
\quad \text{and} \quad
\ell(w)=\ell(x)+\ell(m_\mu)+\ell(y).
\end{align*}
In particular,
\[
\ell(w)-2C\leq \ell(m_\mu)\leq \ell(w),
\]
where \(C\) is the length of a longest word in \(F\).
\end{Lemma}

\begin{proof}
We use the (complementary) description of the projective cell in terms of finite left and
right corrections of spherical elements, due to Lusztig--Xi and Xi
\cite{LuXi-canonical-left-cells,Xi-lowest-based-ring-II}.
In the form needed here, this says that there exist finite sets
\(F_{\mathrm L},F_{\mathrm R}\subset W\), depending only on the affine type,
such that every element \(w\) in the projective cell admits a factorization
\[
w=xm_\mu y
\]
for some \(x\in F_{\mathrm L}\), \(y\in F_{\mathrm R}\), and some dominant
weight \(\mu\), with
\[
\ell(w)=\ell(x)+\ell(m_\mu)+\ell(y).
\]
Thus, after setting
$F:=F_{\mathrm L}\cup F_{\mathrm R}$,
we obtain the claimed form \(w=xm_\mu y\) with \(x,y\in F\).
Finally, if
$C:=\max\{\ell(z)\mid z\in F\}$,
then the length additivity gives
\[
\ell(m_\mu)\leq \ell(w)
\text{ and }
\ell(m_\mu)
=
\ell(w)-\ell(x)-\ell(y)
\geq \ell(w)-2C.
\]
This proves the claim.
\end{proof}

Applying \autoref{lem:projective-close-to-satake} to the alcoves \(r_l\) met
by the ray, we obtain dominant weights \(\mu_l\) such that
\[
\ell(m_{\mu_l})\in \Theta(l).
\]

Now, given $w\in W$ we consider \(m_\lambda \in W_f w W_f\), the longest representative of
this double coset. This element is two-sided longest, and so
\(b_{m_\lambda}\in \Hecke_{\mathrm{sph}}\) by 
\autoref{L:TwoSided}. This is associated to \(w\) in a canonical way,
and thus we aim to relate \(\ell(w)\) to \(\ell(m_\lambda)\), and
\(\nu^\delta(b_w)\) to \(\nu^\delta(b_{m_\lambda})\). We do so in the
following lemma.

\begin{Lemma}\label{lem:finite factor bound}
Let \(F\) and \(C\) be as above. For every \(w\) in the projective cell, there
is a dominant weight \(\lambda\) such that
\[
\ell(w)\leq \ell(m_\lambda)\leq \ell(w)+2\ell(w_0),
\]
and
\[
2^{-2\ell(w_0)}\nu^\delta(b_{m_\lambda})
\leq
\nu^\delta(b_w)
\leq
K\,\nu^\delta(b_{m_\lambda}),
\]
for some constant \(K>0\), depending only on the affine type.
\end{Lemma}

\begin{proof}
Let \(m_\lambda\) from \(w\) be the longest
element in the double coset \(W_f w W_f\). If we write
\[
m_\lambda=gwh,
\]
where
\[
\ell(m_\lambda)=\ell(g)+\ell(w)+\ell(h),
\]
then clearly
\[
\ell(g),\ell(h)\leq \ell(w_0),
\]
and the length inequality follows.

Now, we recall that
\[
\nu^\delta(b_sb_x)=2\nu^\delta(b_x)
\]
for any simple reflection \(s\in W\) and any \(x\in W\). Furthermore,
\(b_{m_\lambda}\) is a KL summand of \(b_gb_wb_h\). We also observe that
\(b_g\) is a KL summand of
\(b_{s_{i_1}}\dots b_{s_{i_k}}\), where
\(s_{i_1}\dots s_{i_k}\) is a reduced expression for \(g\), and similarly
for \(b_h\). Hence
\[
\nu^\delta(b_g)\leq 2^{\ell(g)}\leq 2^{\ell(w_0)},
\]
and similarly for \(h\). Therefore
\[
\nu^\delta(b_{m_\lambda})
\leq
2^{2\ell(w_0)}\nu^\delta(b_w).
\]
(In fact, this construction makes sense for every \(w\in W\), and the
lower comparison does not use the assumption that \(w\) lies in the
projective cell.)

It remains to prove the upper bound. By
\autoref{lem:projective-close-to-satake}, there is some dominant weight
\(\mu\) such that
\[
w=xm_\mu y,
\]
for \(x,y\in F\), with \(\ell(x),\ell(y)\leq C\). Then \(b_w\) is a KL summand
of \(b_xb_{m_\mu}b_y\), and hence
\[
\nu^\delta(b_w)
\leq
2^{2C}\nu^\delta(b_{m_\mu}).
\]

The Satake elements \(m_\mu\) and \(m_\lambda\) differ by multiplication on the
left and on the right by elements from a fixed finite set. Hence the dominant
weights \(\mu\) and \(\lambda\) differ by a uniformly bounded amount, depending
only on the affine type.
Namely, since \(m_\mu\) and \(m_\lambda\) differ by left and right multiplication by
elements from a fixed finite set, there is a constant \(M>0\), depending only
on the affine type, such that
\[
\|\lambda-\mu\|\leq M.
\]
By \autoref{lem:satake},
\[
\nu^\delta(b_{m_\mu})
=
|W_f|\dim L(\mu),
\quad
\nu^\delta(b_{m_\lambda})
=
|W_f|\dim L(\lambda).
\]
Weyl's dimension formula therefore gives a constant \(D>0\), depending only on
the affine type, such that
\[
\nu^\delta(b_{m_\mu})
\leq
D\,\nu^\delta(b_{m_\lambda}).
\]
Combining the last two estimates gives the desired upper bound after setting
\(K=2^{2C}D\).
\end{proof}

From the length bound, we have that
\[
\ell(m_\lambda)\leq \ell(w)+O(1).
\]
Hence, much like with the family of \(m_{\mu_l}\), we also have a Satake element
\(m_{\lambda_l}\) for each alcove \(w_l\) met by \(r\), such that
\[
\ell(m_{\lambda_l})\in \Theta(l).
\]
Now, we can show \autoref{thm:poly-projective-ray}.
By \autoref{lem:satake},
\[
\nu^\delta(b_{m_{\lambda_l}})
=
|W_f|\dim L(\lambda_l).
\]
Weyl's dimension formula gives
\[
\dim L(\lambda_l)
=
\prod_{\alpha\in \Phi^+}
\frac{\langle \lambda_l+\rho,\alpha^\vee\rangle}
{\langle \rho,\alpha^\vee\rangle}.
\]
Since \(r\) is regular, after applying an element of the finite Weyl group its
direction lies in the interior of the dominant chamber. The dominant weights
\(\lambda_l\) stay within bounded distance of this dominant ray. Therefore
\[
\langle \lambda_l,\alpha^\vee\rangle\in \Theta(l)
\quad
\text{for all }\alpha\in \Phi^+.
\]
Every factor in Weyl's dimension formula grows linearly in \(l\), and so
\[
\dim L(\lambda_l)\in \Theta\bigl(l^{|\Phi^+|}\bigr).
\]

Combining the estimates gives the desired
\[
p_l(r)=\nu^\delta(b_{w_l})
\in \Theta\bigl(l^{|\Phi^+|}\bigr),
\]
and we are done.
\end{proof}

\begin{Remark}
Unfortunately, the constants must depend on such a ray. There is no uniform constant $A$ in general such that $A \cdot \ell(w)^{|\Phi^+|} \leq \nu^\delta(b_w)$. 
\end{Remark}

We define a \textit{singular affine ray} $r(t) = tv + u$ to be a ray such that $v = \pm\omega_i$ for some $i$, and $u$ is a translation vector.

\begin{Remark}
For example, if $W = \widetilde{A_2}$, $r(t) = t\omega_1$ doesn't hit the interior of any alcoves, so $r_l$ as defined for regular rays makes no sense here. However, translating by (an appropriate multiple of) $\omega_1+\omega_2$, referring to \autoref{fig:a2-alcoves}, we have a ray along the cell with the words 
\[id \ra 3 \ra 32 \ra 321 \ra 3213 \ra 32132 \ra 321321 \ra \dots,\]
for which we may again form a sequence of values $\nu^\delta$.
\end{Remark}

We have the following bound, with uniform constants.

\begin{Lemma}\label{lem:singular rays}
There exist constants \(A,B>0\), depending only on the affine type, such that
for every \(w\) in the projective cell,
\[
A\cdot \ell(w)^d
\leq
\nu^\delta(b_w)
\leq
B\cdot \ell(w)^{|\Phi^+|},
\]
where \(d=d(\Phi)\) is at least the rank of $\Phi$, and explicitly given by
\[
\begin{array}{c|ccccccccc}
\Phi & A_n & B_n & C_n & D_n & E_6 & E_7 & E_8 & F_4 & G_2 \\ \hline
d & n & 2n-1 & 2n-1 & 2n-2 & 16 & 27 & 57 & 15 & 5 \\ \hline
|\Phi^+| & \frac{n(n+1)}{2} & n^2 & n^2 & n(n-1) & 36 & 63 & 120 & 24 & 6
\end{array}
.
\]
(Here $n\geq 2$ or $n\geq 4$ for types $B/C$ and $D$.)
\end{Lemma}

\begin{proof}
The upper bound already appears in the proof of
\autoref{lem:finite factor bound}.

For the lower bound, let \(m_\lambda\) be the longest representative of the
double coset \(W_f w W_f\), with \(\lambda\in Q^+\). By
\autoref{lem:finite factor bound}, after changing constants depending only on
the affine type, it is enough to prove
$\dim L(\lambda)\geq C\cdot \ell(m_\lambda)^r$
for some \(C>0\).

Write
$\lambda=\sum_{i=1}^r a_i\omega_i$
with \(a_i\geq 0\), and put \(S=\sum_i a_i\). The usual length formula for
dominant translations gives
\[
\ell(m_\lambda)
=
\ell(w_0)+\sum_{\alpha\in\Phi^+}\langle \lambda,\alpha^\vee\rangle
=
\ell(w_0)+\sum_i a_i
\sum_{\alpha\in\Phi^+}\langle \omega_i,\alpha^\vee\rangle,
\]
and we are left with a purely representation-theoretical problem: namely, by Weyl's
dimension formula, it suffices to prove
\[
\dim L(\lambda)=
\prod_{\alpha\in\Phi^+}
\frac{\langle {\textstyle\sum_{i=1}^r} a_i\omega_i+\rho,\alpha^\vee\rangle}
{\langle \rho,\alpha^\vee\rangle}\geq C'\cdot (1+S)^r\text{ for some }C'>0.
\]
To this end, choose \(k\) such that \(a_k=\max_j a_j\).
For every positive root \(\alpha\) with
\(\langle \omega_k,\alpha^\vee\rangle>0\), the corresponding factor is bounded
below by a positive constant times \(1+a_k\) since
$\langle {\textstyle\sum_{i=1}^r} a_i\omega_i+\rho,\alpha^\vee\rangle
\geq
a_k\langle \omega_k,\alpha^\vee\rangle+\langle \rho,\alpha^\vee\rangle$.
Hence
\[
\dim L(\lambda)
\geq
C_k(1+a_k)^{d_k},
\text{ where }
d_k=\#\{\alpha\in\Phi^+\mid \langle \omega_k,\alpha^\vee\rangle>0\}.
\]
It remains only to note that \(d_k\geq r\). Indeed, for every simple root
\(\alpha_j\), the connected subdiagram spanned by the path from \(k\) to \(j\)
has a highest root whose support contains \(\alpha_k\). These \(r\) positive
roots are distinct and are all counted by \(d_k\).
Therefore, since \(a_k\geq S/r\),
\[
\dim L(\lambda)
\geq
C_k\cdot (1+a_k)^r
\geq
C'\cdot (1+S)^r.
\]
Finally, the numbers for $|\Phi^+|$ are well-known, and if $I$ denotes the vertex set of the finite Dynkin diagram of the respected type, the displayed values of \(d\) are obtained from
\[
d=\min_{k\in I}d_k=\min_{k\in I}\big(|\Phi^+|-|\Phi^+_{I\setminus\{k\}}|\big),
\]
by deleting one node at a time from the Dynkin diagram and taking the minimum.
\end{proof}

We will see explicit examples, with stronger results in specific cases, in
\autoref{sec:polaa} and \autoref{sec:pola2} below. Outside of the projective cell
things are much more delicate, see \autoref{S:C2} for a short discussion.

\subsection{General bounds for \texorpdfstring{$\nu_{n}$}{nun}}\label{sec:generalasymptotics}

This section will deal with coarse asymptotic bounds for
$\nu_n(b_w)=\nu(b_w^n)$, as $n$ increases. The following is the most general
result currently available, and is fully in the projective cell.

\begin{Theorem}\label{thm:general-lower-upper}
Let $W$ be an affine Weyl group. Assume that $w \in W$ lies in the projective cell. Then there exists $C=C_w>0$
such that for all $n\geq 1$,
\[
C\cdot n^{-|\Phi^+|}\cdot \beta^n\leq \nu_n\leq\beta^n.
\]
\end{Theorem}

The proof of \autoref{thm:general-lower-upper} splits into a few lemmas. We note that \autoref{lem:upper-bound} - \autoref{lem:length-support} apply to any arbitrary Coxeter group, however the above theorem is only for affine Weyl groups, where we may take advantage of associated root systems.

First, the upper bound.

\begin{Lemma}\label{lem:upper-bound}
Let $U$ be an arbitrary Coxeter group. For any KL-positive element $X=\sum_w c_wb_w \in \Hecke(U)$, one has
\[
\nu(X)\leq \nu^\delta(X).
\]
Furthermore, 
\[
\nu_n\leq\beta^n.
\]
In fact, as a consequence of Fekete's lemma, we have 
\[
\limsup_{n\to\infty}\sqrt[n]{\nu_n} = \lim_{n\to\infty}\sqrt[n]{\nu_n}\leq\beta.
\]
\end{Lemma}

\begin{proof}
At $v=1$, every KL basis element expands in the standard basis with
nonnegative coefficients. In particular $\nu^\delta(b_w)\ge 1$, so
\[
\nu^\delta(X)
=
\sum_w c_w\nu^\delta(b_w)
\geq
\sum_w c_w
=
\nu(X).
\]

This proves the first statement. 

Now, by \autoref{nu_del ringhom},
\[
\nu^\delta(b_w^n)
=
\bigl(\nu^\delta(b_w)\bigr)^n
=
\beta^n.
\]

Hence, \[\nu_n = \nu(b_w^n) \leq \nu^\delta(b_w^n) = \beta^n.\]
\end{proof}

Next, the lower bound.

\begin{Lemma}\label{lem:lower-by-standard-sum}
Let $U$ be an arbitrary Coxeter group, and $0\neq X=\sum_{z\in W}a_zb_z \in \Hecke(U)$ be KL-positive. Then,
\[
\nu(X)\geq
\frac{\nu^\delta(X)}
{\max_{w:a_w\neq 0}\nu^\delta(b_w)}.
\]
\end{Lemma}

\begin{proof}
Expand $X$ in the standard basis and sum coefficients:
\[
\nu^\delta(X)
=
\sum_w a_w\nu^\delta(b_w)
\leq
\left(\max_{w:a_w\neq 0}\nu^\delta(b_w)\right)\sum_w a_w.
\]
We are done.
\end{proof}

\begin{Lemma}\label{lem:length-support}
Let $U$ be an arbitrary Coxeter group, $b_w \in \Hecke(U)$, and
\[
b_w^n=\sum_{x\in W}a_x^{(n)}b_x,
\quad
a_x^{(n)}\in \mathbb Z_{\geq 0}.
\]
If $a_x^{(n)}\neq 0$, then
\[
\ell(x)\leq n\ell(w).
\]
\end{Lemma}

\begin{proof}
Every standard basis term appearing in $b_w$ has length at most $\ell(w)$,
hence every standard basis term appearing in $b_w^n$ has length at most
$n\ell(w)$. If $a_x^{(n)}\neq 0$, then $\delta_x$ appears in the standard
expansion of $b_x$ with coefficient $1$. Since all standard expansions are
nonnegative at $v=1$, this contribution cannot cancel. Therefore $\delta_x$
appears in the standard expansion of $b_w^n$, and so $\ell(x)\leq n\ell(w)$.
\end{proof}

\begin{proof}[Proof of \autoref{thm:general-lower-upper}]
We are now in the situation where $U = W$ is an affine Weyl group. The upper bound is \autoref{lem:upper-bound}. For the lower bound, apply
\autoref{lem:lower-by-standard-sum} to \(X=b_w^n\):
\[
\nu_n\geq
\frac{\nu^\delta(b_w^n)}
{\max_{z\in\operatorname{Supp}_{\mathrm{KL}}(b_w^n)}
\nu^\delta(b_z)}
=
\frac{\beta^n}
{\max_{z\in\operatorname{Supp}_{\mathrm{KL}}(b_w^n)}
\nu^\delta(b_z)}.
\]
Since \(w\) lies in the projective cell, and the span of the projective cell
is a two-sided ideal, every \(z\) in the KL-support of \(b_w^n\) also lies in
the projective cell. By \autoref{lem:length-support},
\[
\ell(z)\leq n\ell(w).
\]
Now apply \autoref{lem:singular rays}. This gives
\[
\nu_n
\geq
B^{-1}\bigl(n\ell(w)\bigr)^{-|\Phi^+|}\beta^n.
\]
Since \(w\) is fixed, this gives the claimed lower bound after changing the
constant.
\end{proof}

\begin{Remark}
The exponent \(|\Phi^+|\) in \autoref{thm:general-lower-upper} is often not optimal. For example,
if only some fundamental weight coordinates grow linearly with \(n\), the
degree detected by Weyl's dimension formula is the number of positive roots
involving at least one of the corresponding simple roots. In rank \(1\) and
rank \(2\), one expects substantially sharper statements in many cases. 
See \autoref{sec:sl2}, \autoref{sec:sl3} and \autoref{S:C2} below.
\end{Remark}

\subsection{Elements accessible via Satake}\label{sec:generalpower}

We now apply the spherical formalism from \autoref{sec:satake}. We use the
same symbol \(\nu\) for the sum of KL coefficients in $\Hecke$ and for the sum
of simple multiplicities in \(\Rep G\). The latter is also the number of
indecomposable summands, since $\Rep G$ is semisimple.

\begin{Theorem}[One-sided longest]\label{prop:satake-bounds}
Assume that \(b_we=b_w\), that \(w\) lies in the projective cell, and $w\neq w_0$.
Then there exist constants \(C_1(w),C_2(w)>0\) such that for all \(n\gg 1\),
\[
C_1(w)\cdot n^{-|\Phi^+|}\cdot\beta^n
\leq
\nu(b_w^n)
\leq
C_2(w)\cdot n^{-|\Phi^+|/2}\cdot\beta^n .
\]
\end{Theorem}
\begin{proof}
We first compare \(b_w^n\) with its spherical image. Let \(x\in \Hecke\) have
nonnegative KL expansion. Then \(b_{w_0}x\) again has nonnegative KL
expansion, and each product \(b_{w_0}b_z\) contains at least one KL summand.
Therefore
\[
\nu(b_{w_0}x)\geq \nu(x),
\]
and hence
\[
\nu(ex)\geq \frac{1}{|W_f|}\nu(x).
\]
Applying this to \(x=b_w^n\), and using
\autoref{lem:satake-one-sided}, gives
\begin{gather}\label{eq:one-sided-satake-comparison}
\nu(b_w^n)
\leq
|W_f|\nu(eb_w^n)
=
|W_f|\nu\bigl((eb_w)^n\bigr)
=
|W_f|\nu(V_w^{\otimes n}).
\end{gather}

By \autoref{lem:satake-one-sided}, \(\dim(V_w)=\beta\). 
Moreover, $V_w$ is faithful: since \(G\) is complex adjoint simple, every nontrivial finite dimensional
\(G\)-representation is faithful (the kernel is a closed normal subgroup, and
any finite normal subgroup of a connected algebraic group is central; but
\(G\) is adjoint, so its center is trivial). We can then use
\cite[Theorem 2.5]{CEO19}, applied with \(s=0\), and get
\[
\nu(V_w^{\otimes n})
\sim
C_w\cdot n^{-|\Phi^+|/2}\cdot \beta^n
\]
for some \(C_w>0\). This proves the upper bound.

The lower bound is exactly \autoref{thm:general-lower-upper}, applied to
\(b_w\), since \(\nu^\delta(b_w)=\beta\).
\end{proof}

\begin{Remark}\label{rem:uniform-false}
A natural guess would be that the finite multiplicative correction appearing in the above
comparison \autoref{eq:one-sided-satake-comparison} is uniformly bounded in terms of the affine type. This is false in general - there are cases where no such multiplicative correction exists in the lower bound.

Indeed, cf. \autoref{S:C2}, in affine type \(\widetilde C_2\), with affine node \(3\), experiments
inside the subregular cell, containing elements with a unique reduced
expression, indicate that for
\[
w_k=3(1213)^k
\]
one has
\[
\ell(w_k)=4k+1,
\quad
\nu(b_{w_k}b_{w_0})=6k-4
\quad(k\geq 2).
\]
Thus this quantity grows linearly with \(k\).

On the other hand, boundedness holds in some important regimes; for example, in
type \(\widetilde A_2\) for the beyond-the-wall family treated in
\autoref{prop:uniform6}. This allows us to improve the theorem in certain cases, e.g. replacing $\Phi$ by $\Phi/2$.
\end{Remark}

\begin{Theorem}[Two-sided longest]\label{prop:satake-two-sided}
Assume that $w$ is two-sided longest (so that \(b_w\in e\Hecke e\)), and $w\neq w_0$. Then
\[
\nu(b_w^n)
\sim
C_w\cdot n^{-|\Phi^+|/2}\cdot\beta^n
\]
for some constant \(C_w>0\).
\end{Theorem}

\begin{proof}
Since \(b_w\in e\Hecke e=\Hecke_{\mathrm{sph}}\), Satake associates to \(b_w\)
a representation \(V_w\in\Rep G\). Under this identification,
\(\nu^\delta\) is the dimension homomorphism, so
\[
\dim(V_w)=\nu^\delta(b_w)=\beta.
\]
Since Satake is monoidal, \(b_w^n\) corresponds to \(V_w^{\otimes n}\). Hence
\[
\nu(b_w^n)=\nu(V_w^{\otimes n}).
\]
The claimed asymptotic now follows, as for one-sided longest elements, directly from
\cite[Theorem 2.5]{CEO19}.
\end{proof}

Thus, for two-sided longest elements, the growth problem reduces directly to
tensor powers in \(\Rep G\), with no need for the one-sided sandwich argument.

\section{The case of SL2, or affine type \texorpdfstring{$A_1$}{A1}}\label{sec:sl2}

We now consider the Coxeter group of affine $A_1$-type. This leads to adjoint group
\(G=PGL2\), which is isomorphic
to \(SO3\). Let the affine Coxeter generators be $s$ and $t$, so that
\[
\widetilde A_1=\langle s,t\mid s^2=t^2=id\rangle.
\]

Every element is represented by one of
\[
(st)^k,\quad (st)^ks,\quad (ts)^k,\quad (ts)^kt,
\]
where $k\in \N$. By symmetry between $s$ and $t$, it suffices to consider the
two families
\[
w_k^{\mathrm{pal}}:=(st)^ks,
\quad
w_k^{\mathrm{np}}:=(ts)^k.
\]
These are the palindromic and nonpalindromic words, respectively.

\subsection{Growth of polynomials}\label{sec:polaa}
We start by the general question from Section~\ref{sec:poly} for this example.
For a ray $r$ starting at the origin in $\R\cong \R\otimes_{\R}Q$, we are thus interested in
\[
p_l:=p_l(r)=\nu^\delta(b_w),
\]
where $w$ is the alcove of length $l$ met by the ray. We study this as a
function of the length.
There are only two rays, namely $\R_{\geq 0}$ and $\R_{\leq 0}$,
for which $p_l$ is the sequence
\[
\big(\dots,\nu^\delta(b_{w_k^{\mathrm{pal}}}),
\nu^\delta(b_{w_k^{\mathrm{np}}}),
\nu^\delta(b_{w_{k+1}^{\mathrm{pal}}}),
\nu^\delta(b_{w_{k+1}^{\mathrm{np}}}),
\nu^\delta(b_{w_{k+2}^{\mathrm{pal}}}),
\nu^\delta(b_{w_{k+2}^{\mathrm{np}}}),
\dots
\big),
\]
or its reverse. 

\begin{Proposition}
For either ray in affine type $\widetilde A_1$, one has $p_0=1$ and
\[
p_l=2l
\quad
(\Rightarrow p_l\sim 2l).
\]
\end{Proposition}

\begin{proof}
For the infinite dihedral group, all the KL polynomials at $v=1$ are equal
to $1$, see e.g. \cite{duCl-positivity-finite-hecke}. Hence
\[
b_w=\sum_{x\leq w}\delta_x
\]
at $v=1$, and therefore
\[
\nu^\delta(b_w)=\#\{x\in W\mid x\leq w\}.
\]
If $\ell(w)=l$, then the Bruhat interval $[e,w]$ consists of the identity, $w$ itself,
together with the two alternating words of each length $1,\dots,l-1$. Thus
\[
\#[e,w]=2l.
\]
Therefore $p_l=2l$, as claimed.
\end{proof}

\subsection{Growth of tensor powers}

Explicitly, here $W_f=\{id,s\}$. Thus the spherical idempotent is
$b_s/2\in \Hecke_{\mathrm{sph}}$. A basis of $\Hecke_{\mathrm{sph}}$ is given by $b_w$ such that $w$ is a longest representative of its coset $W_f\setminus W / W_f$, as in \autoref{L:TwoSided}. In this case, it is easily checked that these are all palindromic words starting (and thus ending) with $s$. Notate the $w$ of length $2k+1$ satisfying $b_w \in \Hecke_{\mathrm{sph}}$ by $w_k^{\mathrm{pal}}$.

The simple PGL2-representation attached to
$w=w_k^{\mathrm{pal}}$ is $L(2k)$, of dimension
\[
\dim L(2k)=2k+1.
\]
For example, $w=sts$ corresponds to the three-dimensional representation of
\(PGL2\), or equivalently to the defining representation of
$SO3$.

\begin{Theorem}\label{thm:a1exact}
For every $w\neq e,s,t$, there exists a constant $C_w>0$ such that
\[
\nuop\bigl(b_{w}^n\bigr)\sim C_w\cdot n^{-1/2}\cdot\beta^n
\]
More precisely:
\begin{enumerate}
\item For the palindromic family,
\[
b_{(st)^ks}\longmapsto 2\cdot L(2k)
\]
under Satake, and hence
\[
\nuop\bigl(b_{(st)^ks}^n\bigr)
\sim
\nuop\bigl((2\cdot L(2k))^{\otimes n}\bigr)
\sim
\sqrt{3/\big((2k^2+1)\pi\big)}\cdot n^{-1/2}\cdot (4k+2)^n.
\]

\item For the nonpalindromic family,
\[
\nuop\bigl(b_{(ts)^k}^n\bigr)
\sim
\sqrt{3/\big(2k(k+2)\pi\big)}\cdot n^{-1/2}\cdot (4k)^n.
\]
\end{enumerate}
In particular, in both families the polynomial correction is the rank-one
factor $n^{-1/2}$.
\end{Theorem}

\begin{proof}
We treat the two families separately.

For the palindromic family, one is in the genuinely spherical case. Under
Satake, writing a representation to mean its isomorphism class in the
Grothendieck ring, one has
\[
b_{(st)^ks}\longmapsto 2\cdot L(2k),
\]
where $L(2k)$ denotes the simple $PGL2$-module of highest weight $2k$.
Hence the claim follows directly from the rank-one tensor power asymptotics in
\cite[Theorem 2.5]{CEO19}, namely
\[
\nuop\bigl((2\cdot L(2k))^{\otimes n}\bigr)
\sim
\sqrt{3/\big((2k^2+1)\pi\big)}\cdot n^{-1/2}\cdot (4k+2)^n.
\]

For the nonpalindromic family, one is in the one-sided longest element
situation. Using $|W_f|=2$, the spherical idempotent $e=b_s/2$, and the fusion
rules
\begin{align*}
\frac{b_s}{2}b_{(ts)^k}
&=
\frac{b_{s(ts)^k}+b_{s(ts)^{k-1}}}{2},\\
b_{(ts)^k}\frac{b_s}{2}
&=
b_{(ts)^k},
\end{align*}
which are well known and can, for example, be found in
\cite[Section 3]{Tu-sandwich-cellular}, one obtains
\[
\nuop\Bigl(\frac{b_sb_w}{2}\Bigr)=\nuop(b_w)
\quad\text{for }w\notin\{s,t\}.
\]
Moreover, \autoref{lem:satake-one-sided} directly gives
\[
\Bigl(\frac{b_s}{2}b_{(ts)^k}\Bigr)^n
=
\frac{b_s}{2}b_{(ts)^k}^n,
\]
since words ending in $s$ are at least one-sided longest. The corresponding
$PGL2$-representation is thus
\[
L(2k)\oplus L(2k-2).
\]
Applying again \cite[Theorem 2.5]{CEO19} yields
\[
\nuop\bigl(b_{(ts)^k}^n\bigr)
\sim
\sqrt{3/\big(2k(k+2)\pi\big)}\cdot n^{-1/2}\cdot (4k)^n.
\]
This proves the theorem.
\end{proof}

\section{The case of SL3, or affine type \texorpdfstring{$A_2$}{A2}}\label{sec:sl3}

We now consider the Coxeter group of affine $A_2$-type. This leads to adjoint group
\(G=PGL3\). Let the finite Coxeter generators be $1,2$, and the affine reflection be $3$. Our presentation for $\widetilde{A_2}$ is
\[
\widetilde A_2=\langle 1,2,3\mid 121=212,\ 232=323,\ 131=313,\ 1^2=2^2=3^2=id\rangle.
\]
As in type \(\Aone\), the symbols \(1\), \(2\) and \(3\) are exchangeable, and
below we will only treat things up to that symmetry. By \autoref{ex:finite}, we will again 
assume that $w\in W$ contains at least one of each of the generators $1,2,3$. 

\begin{Remark}
Despite the results being the same no matter how we label the generators in this situation, the choice of finite generators above must be made so as to invoke the Satake isomorphism. 
\end{Remark}

\subsection{SL3 geometry}\label{sec:SL3-geo}

It is useful to keep the usual alcove picture of \(\Atil\) in mind, cf. \autoref{fig:a2-alcoves}. 

\begin{Remark}
All of our alcove pictures are from Gibson's brilliant online illustrations on Lievis \url{https://www.jgibson.id.au/lievis}.
\end{Remark}

The affine
Coxeter arrangement is the triangular tiling of the plane. The alcoves are the
small triangles, and crossing a wall labeled \(i\in\{1,2,3\}\) corresponds to
right multiplication by the simple reflection \(i\). Thus a reduced word records
a minimal gallery from the fundamental alcove to the alcove indexed by the
corresponding element of \(W\).

\begin{figure}[ht]\label{fig:A2}
\centering
\includegraphics[height=6cm]{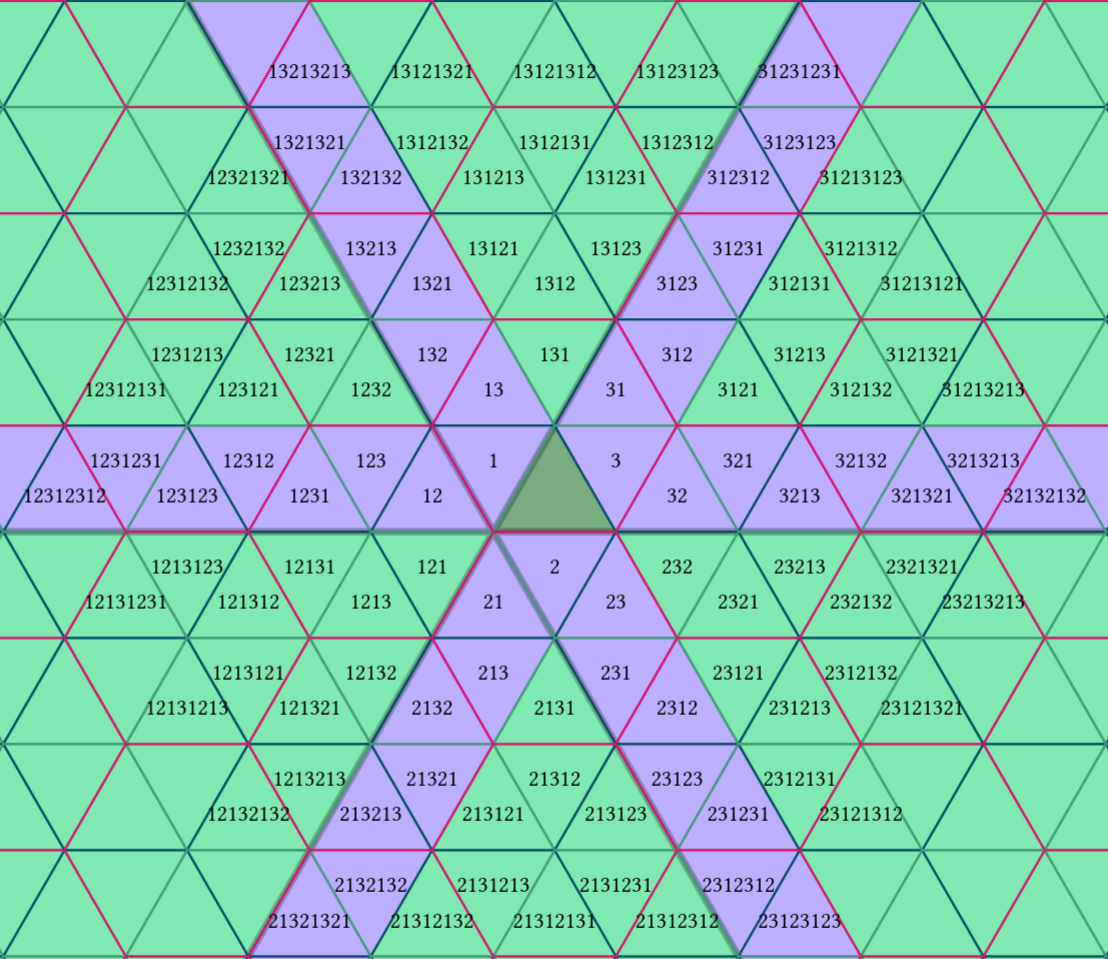}
\includegraphics[height=6cm]{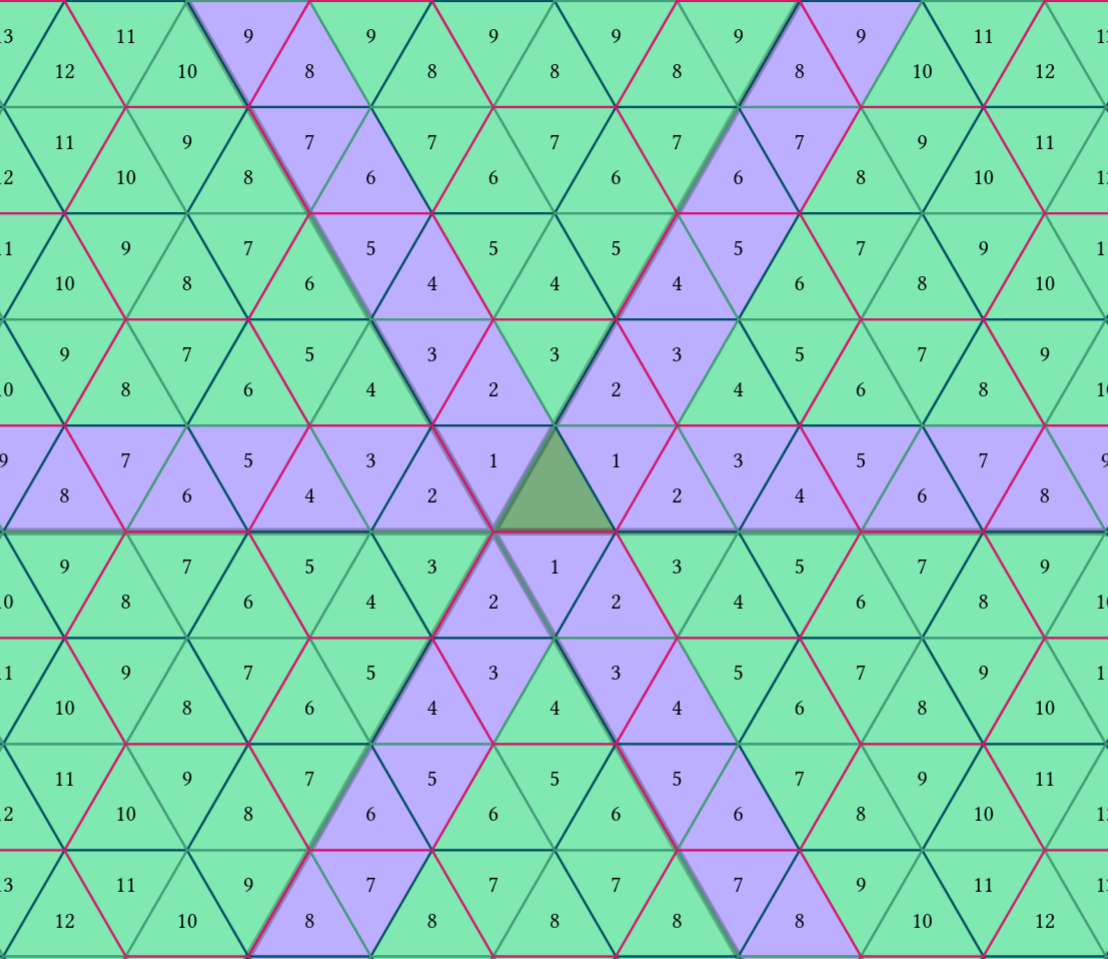}
\includegraphics[height=6cm]{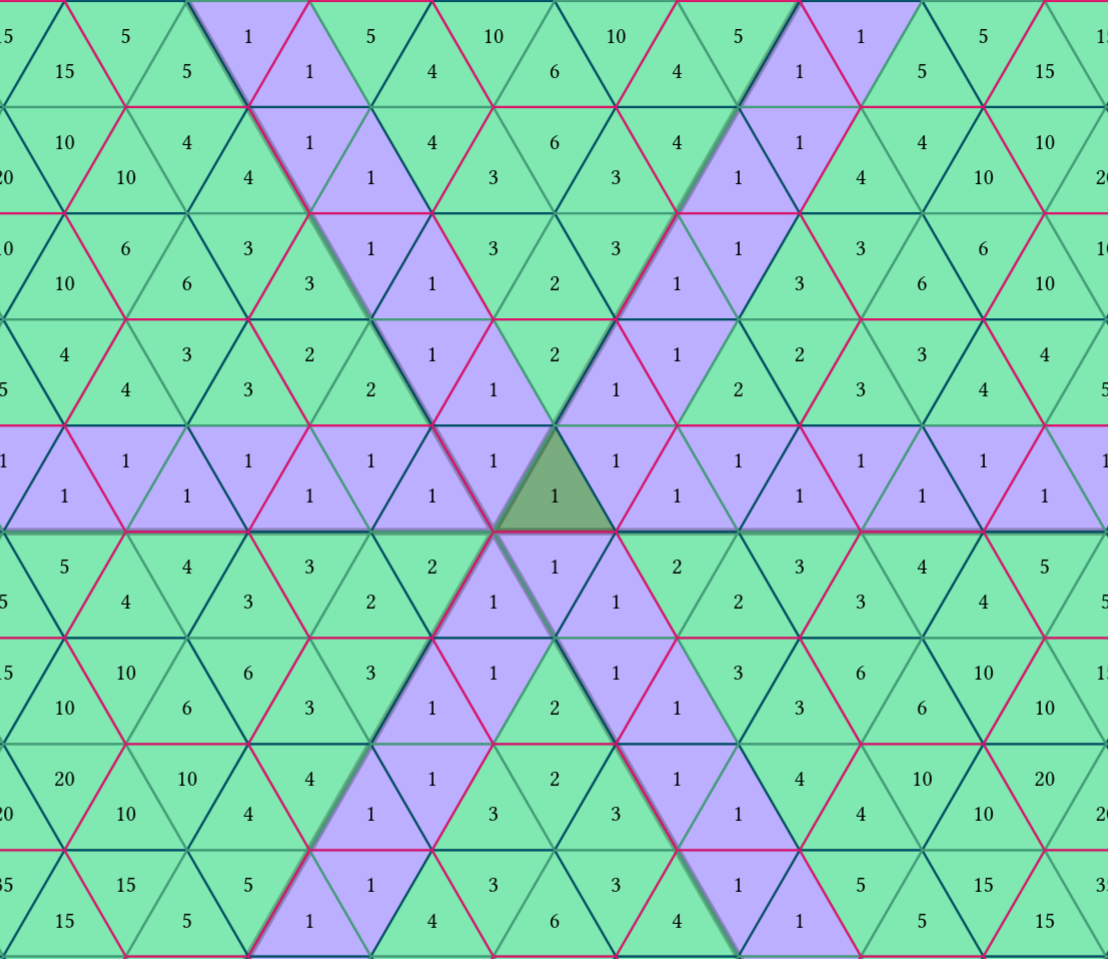}
\includegraphics[height=6cm]{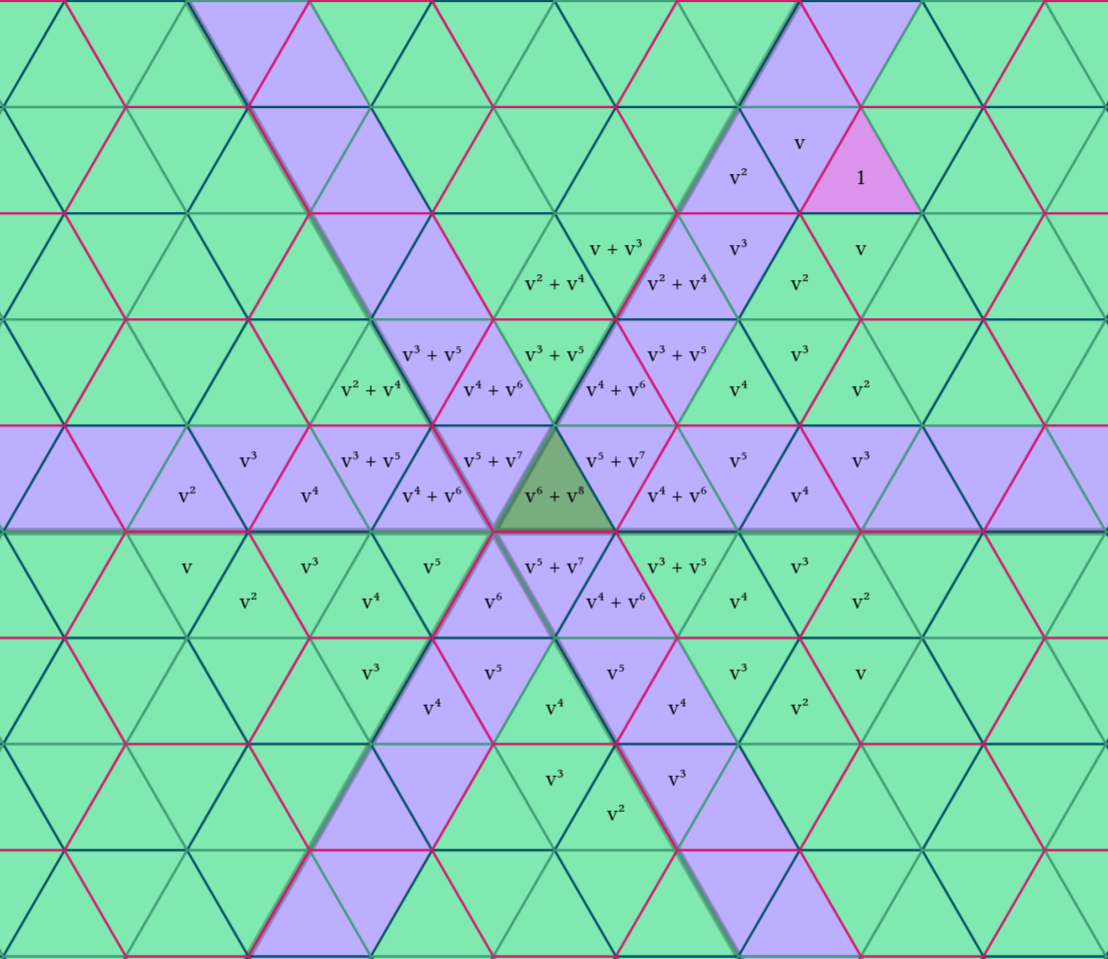}
\caption{The affine \(\widetilde A_2\) alcove picture, with hyperplane colors
\(1=\) green, \(2=\) red, and \(3=\) blue. In addition to the alcoves, the
picture shows the two-sided KL cells: the shaded central triangle is the trivial
cell, i.e. the fundamental alcove; the purple belt is the cyclic wall; and the
green cones form the projective cell. Each alcove is labeled by information
about the corresponding element \(w\in W\): northwest gives a reduced expression,
northeast the number of reduced expressions, southwest the length, i.e. the
minimal number of walls crossed from the fundamental alcove, and southeast an
example of a KL polynomial for the shaded alcove.}
\label{fig:a2-alcoves}
\end{figure}

There are three geometric regions that we use below. The green shaded region is referred to as \textbf{projective cell, big cell}, and the area \textbf{beyond-the-wall}. The purple shaded region is the region \textbf{beyond-the-wall}, or \textbf{cyclic} section. First, the trivial (central)
cell consists of the fundamental alcove. Second, the `on-the-wall' region (purple belt) consists
of the alcoves obtained by adding generators cyclically (geometrically, this is following a wall). These are represented,
up to cyclic rotation and reversal, by the words
\[ 
x_1 = 1,\quad x_2 = 12,\quad x_3 = 123,\quad x_4 = 1231,\quad x_5 = 12312,\quad \dots,
\]
and in general, for $j\geq 1$, define
\[
x_j:=123123\cdots
\]
to be the cyclic word of length $j$, starting in $1$ and continuing in the order
\[
1\to 2\to 3\to 1\to 2\to 3\to \cdots.
\]
From here on, we will refer to these as ``cyclic" elements.

Third, the projective, or beyond-the-wall, cell (green cones) consists of the alcoves
lying beyond this wall region. This is the large two-sided cell relevant for the
asymptotic growth below.
We follow the notation of \cite{LiPa}, labeling simple reflections modulo $3$, and
write
\[
\theta(m,n)=123\cdots(2m+1)(2m+2)(2m+1)2m\cdots(2m-2n+1)
\]
for some beyond-the-wall words. As seen in \cite{LiPa}, these describe the projective cell up to slight variations. To ease notation, we will use \(y_j\) and \(z_j\) for the projective elements adjacent to the cyclic wall, as displayed in \autoref{fig:a2-words},
where $j \geq 4$. From here, we will refer to all beyond-the-wall elements as `projective'.

The conventions are illustrated in \autoref{fig:a2-words}. The same illustration shows the Satake elements, for which \autoref{sec:satake} gives the growth rate. These make up $\approx 1/6$ of the whole picture.

\begin{figure}[ht]
\centering
\includegraphics[height=6cm]{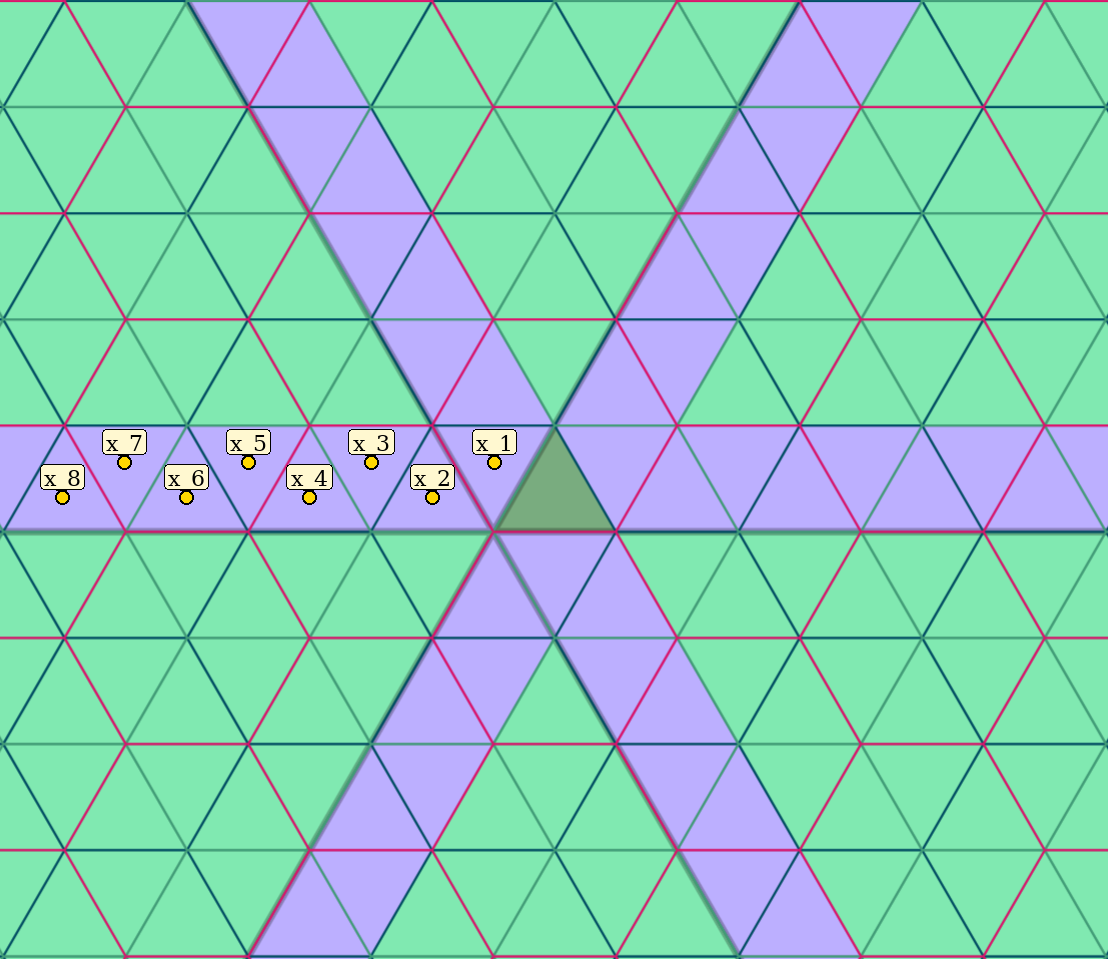}
\includegraphics[height=6cm]{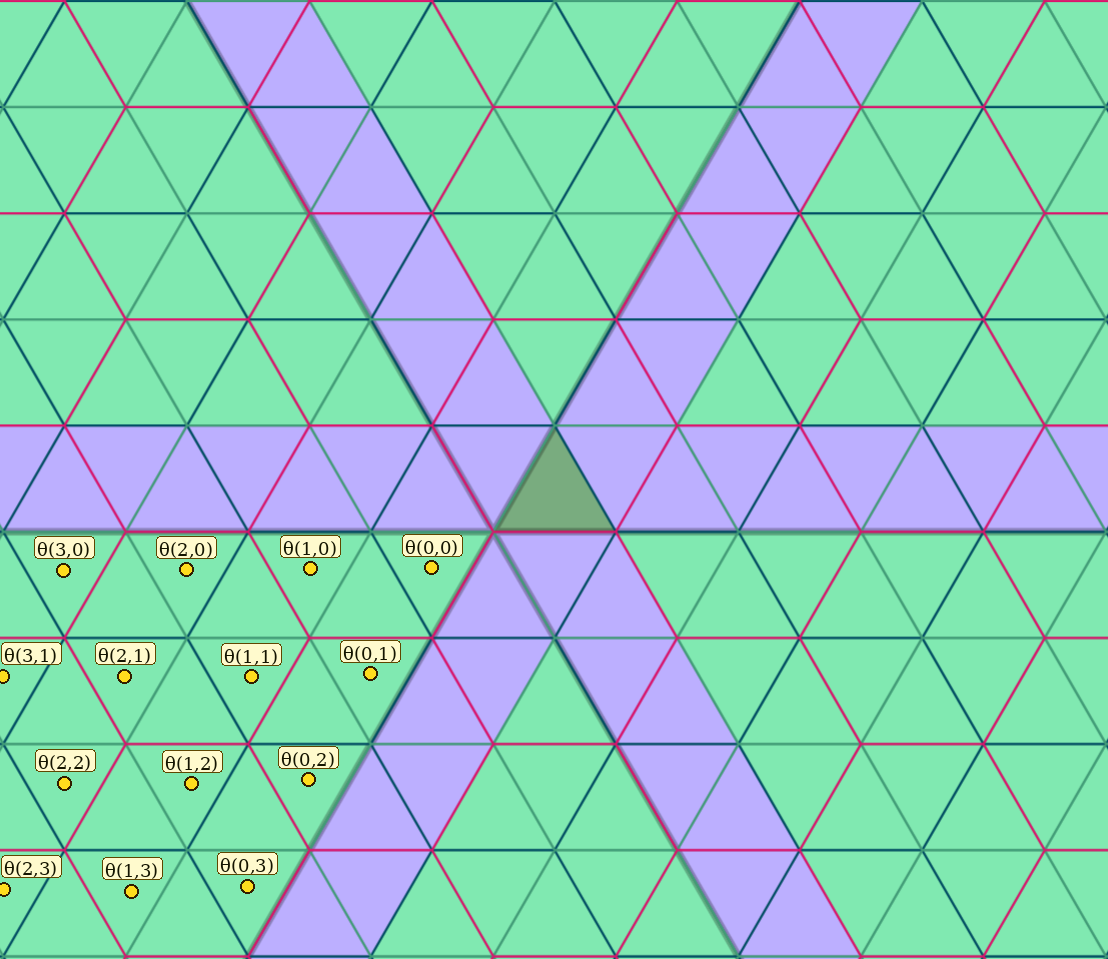}
\includegraphics[height=6cm]{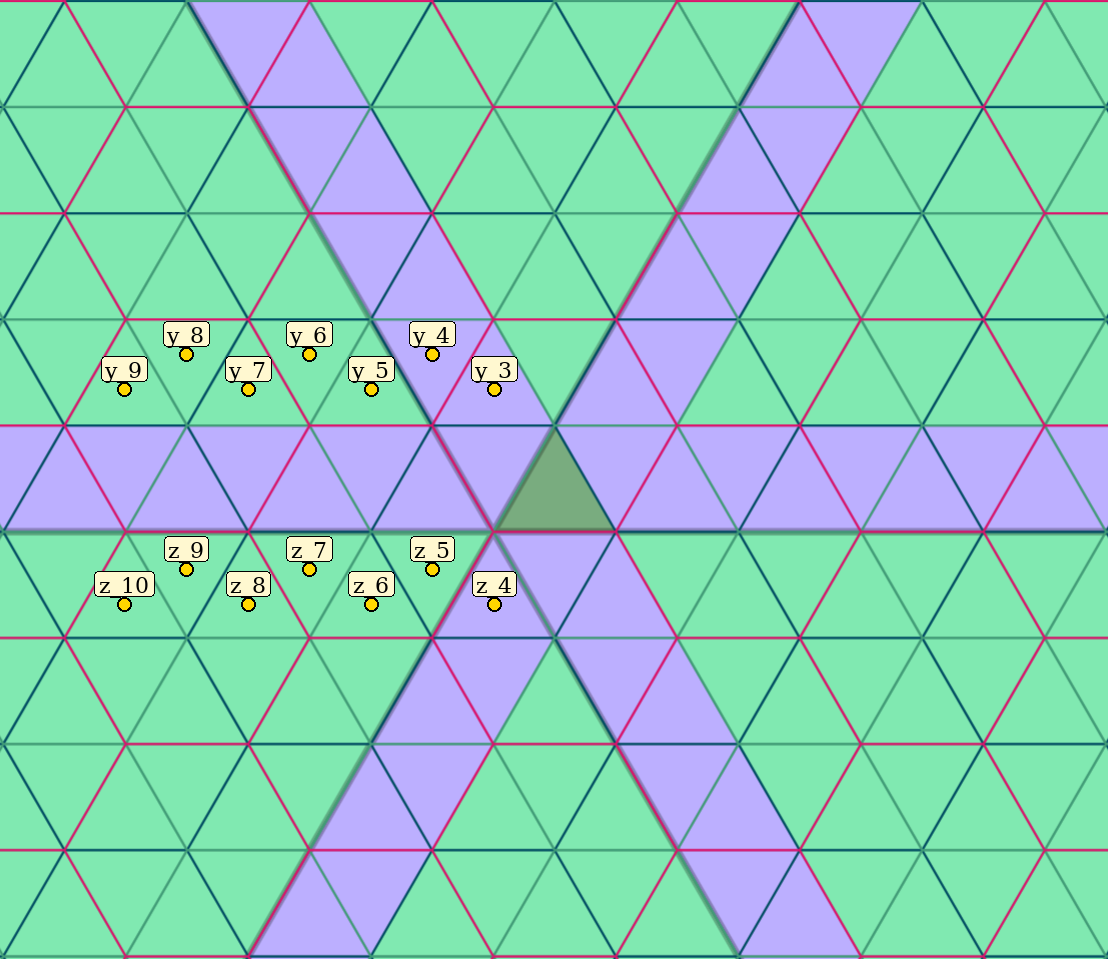}
\includegraphics[height=6cm]{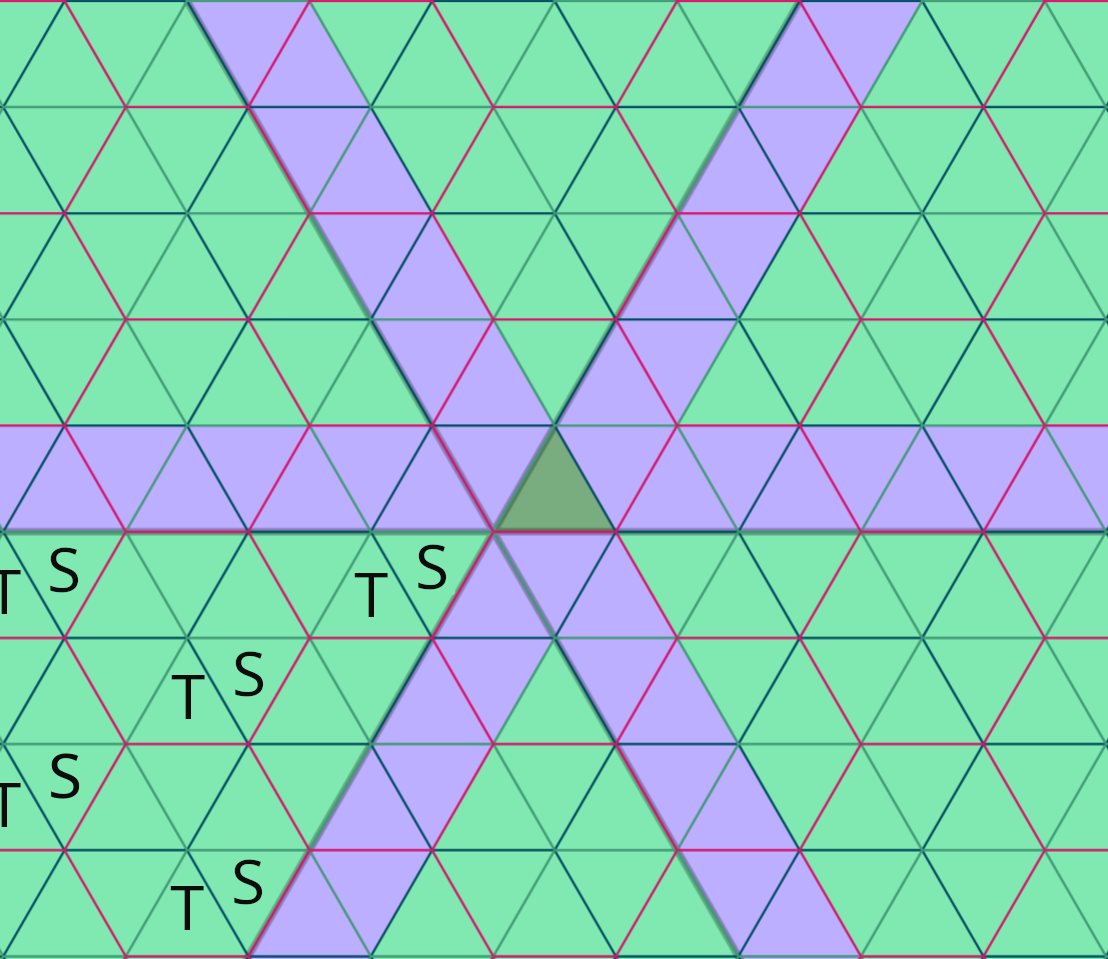}
\caption{The cyclic words (northwest) and the projective words (northeast). At the southeast, for one cone, we mark the translations T of the fundamental alcove; these are one-sided longest. We also mark the two-sided longest elements S, which lie in $\Hecke_{sph}$ and hence can be mapped to representations under the Satake isomorphism. The inner hexagon consists of the elements corresponding to $W_f$; in particular, the longest element $w_0$ is included in the spherical picture, but corresponds to the trivial representation. The other cones are obtained by rotating the
picture and relabeling the simple reflections. Finally, the southwest shows the $y_k,z_k$ as in \autoref{sec:cyclic}. For the alcove conventions, see
\autoref{fig:a2-alcoves}.}
\label{fig:a2-words}
\end{figure}

\subsection{The Coxeter element as a test case}\label{sec:cox}

The following is based on computer calculations that are available at \cite{COT}.
We record the first few powers of the Coxeter element
\[
	a:=b_{123}=b_{x_3}.
\]
In the standard basis, we have
\[b_{x_3} = 1 + \delta_1 + \delta_2 + \delta_3 + \delta_{12} + \delta_{13} + \delta_{23} + \delta_{123},\]
and $\beta = \nu^\delta(b_{x_3}) = 8$. This is the first element $b_w$ along the wall where $w \notin W_f$.
\\
\\

At \(v=1\), the KL expansions for the first five powers are:
\begin{align*}
a
&=
b_{x_3},
\\[0.4em]
a^2
&=
b_{\theta(0,0)3}
+
b_{y_5}
+
b_{x_6},
\\[0.4em]
a^3
&=
b_{y_4[132]}
+
2b_{\theta(0,0)3}
+
2b_{y_5}
+
2b_{y_6}
+
2b_{\theta(0,1)}
+
2b_{\theta(0,1)[132]}
\\
&\quad
+
3b_{y_8}
+
2b_{2\theta(1,1)3}
+
b_{x_9},
\\[0.4em]
a^4
&=
2b_{y_4[132]}
+
6b_{\theta(0,0)3}
+
6b_{y_5}
+
6b_{y_6}
+
10b_{\theta(0,1)}
+
6b_{\theta(0,1)[132]}
\\
&\quad
+
2b_{2\theta(0,2)}
+
2b_{\theta(1,0)2[132]}
+
12b_{y_8}
+
6b_{2\theta(1,1)3}
+
5b_{\theta(1,1)3}
\\
&\quad
+
3b_{y_9}
+
5b_{2\theta(1,2)}
+
3b_{\theta(0,2)2[132]}
+
4b_{\theta(3,0)3}
+
4b_{y_{11}}
+
b_{x_{12}},
\\[0.4em]
a^5
&=
6b_{y_4[132]}
+
22b_{\theta(0,0)3}
+
22b_{y_5}
+
23b_{y_6}
+
48b_{\theta(0,1)}
+
23b_{\theta(0,1)[132]}
\\
&\quad
+
10b_{2\theta(0,2)}
+
10b_{\theta(1,0)2[132]}
+
56b_{y_8}
+
5b_{\theta(0,2)}
+
23b_{2\theta(1,1)3}
\\
&\quad
+
5b_{\theta(1,1)[132]}
+
5b_{y_8[132]}
+
32b_{\theta(1,1)3}
+
12b_{y_9}
+
32b_{2\theta(1,2)}
\\
&\quad
+
12b_{\theta(0,2)2[132]}
+
9b_{\theta(2,1)}
+
14b_{\theta(1,2)}
+
5b_{2\theta(0,3)3}
\\
&\quad
+
9b_{\theta(1,2)[132]}
+
4b_{\theta(0,3)[132]}
+
20b_{\theta(3,0)3}
+
20b_{y_{11}}
\\
&\quad
+
14b_{\theta(3,1)}
+
5b_{y_{12}}
+
9b_{2\theta(2,2)3}
+
5b_{\theta(0,4)[132]}
\\
&\quad
+
6b_{y_{14}}
+
5b_{2\theta(4,1)3}
+
b_{x_{15}}.
\end{align*}
Here we use $[abc]$ to indicate a relabeling of $1,2,3$ to $a,b,c$. The first few numbers $\nu_n$ are
\begin{gather*}
1, 1, 3, 17, 83, 472, 2813, 17665, 114707, 764440, 5197399, 35908465, 251369939, 1779066001,\\
12708857253, 91513539938, 663545968269, 4840510879326, 35500864393015, 261612229520427, \\ 1936121288343995, 14383966729216408, 107235512264419301, 802003121666073689, 
\\ 6015506924947283547, 45240058661666976530, 341065071104834640043, 2577112539257555579253,
\\ 19513690888446080682345, 148044023180636345992540, 1125200468081183529631187, \\
8566519169400932958346667, 65323168570084702584530155.
\end{gather*}
The (normalized) plot is given in \autoref{fig:a2-plots}.

\begin{figure}[ht]
\centering
\includegraphics[height=6cm]{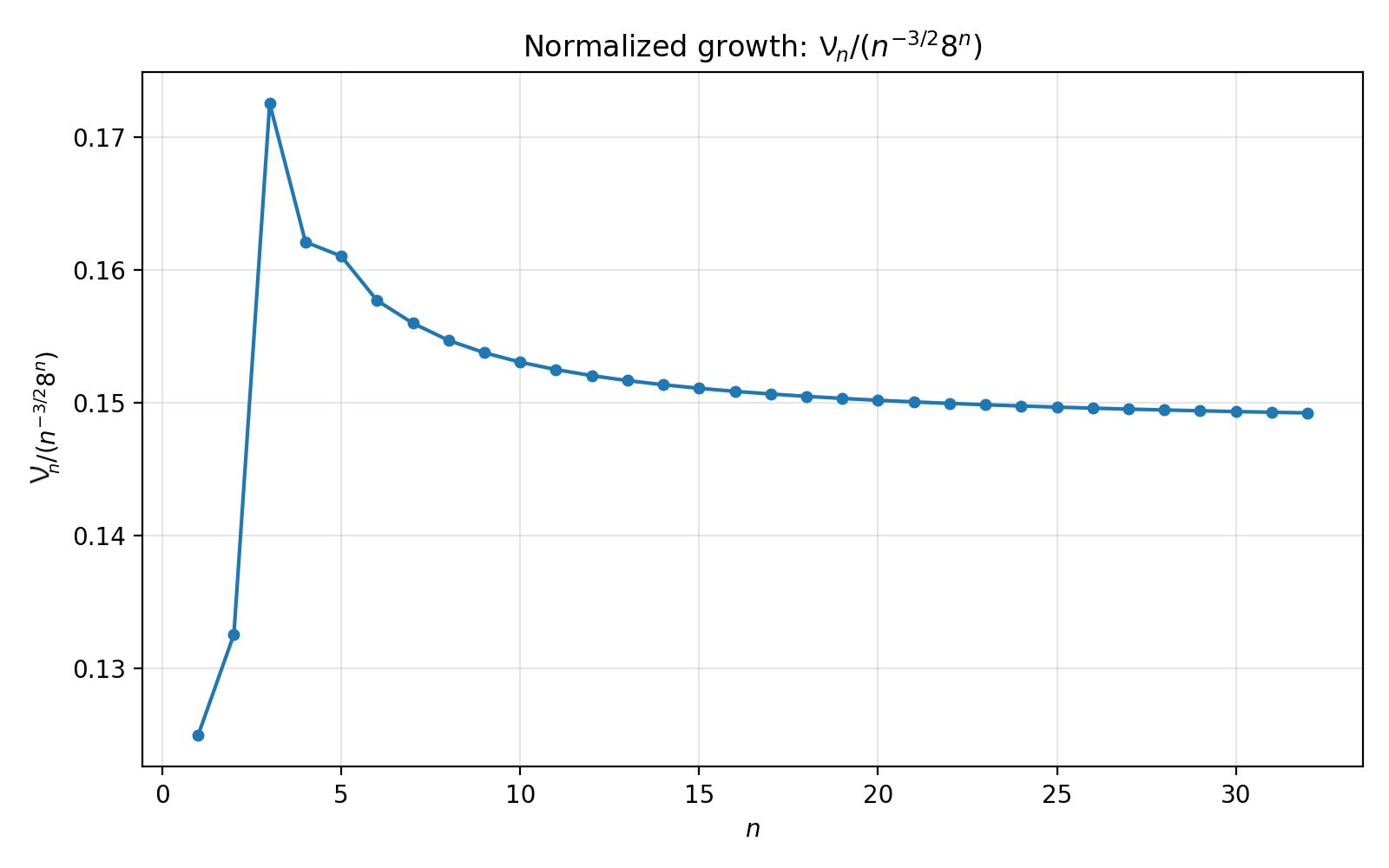}
\caption{The (normalized) sequence $\nu_n$ for $w=123$.}
\label{fig:a2-plots}
\end{figure}

We note two observations:
\begin{enumerate}

\item The powers of $b_{123}$ contain only one summand on the wall. (We will prove this in \autoref{sec:cyclic} in general.)

\item The data suggest that $\nu_n\sim C\cdot n^{-3/2}\cdot 8^n$ (thus, as expected, $\beta=8$). In the computed range, the normalized values $\nu_n/(n^{-3/2}\cdot 8^n)$
lie roughly between \(0.14\) and \(0.15\), so we estimate \(C\) to lie in this range.

\end{enumerate}
It is also worth recording the growth of the standard dimensions along the
wall ray. Specializing the formulas of
\cite[Theorem~1(i), Lemma~1.4]{LiPa} at \(v=1\), one obtains
for \(k\geq 2\),
\[
\nu^\delta(b_{x_{2k+1}})=6k^2+2,
\]
while, for \(k\geq 3\),
\[
\nu^\delta(b_{x_{2k}})=6k^2-6k+4.
\]
Thus the KL polynomials themselves grow quadratically along the wall ray
through the Coxeter element, in contrast with the generic growth in the
projective cell, which is cubic by
\autoref{thm:poly-projective-ray}.

\begin{Remark}
We stress the following distinction between the two asymptotic growth problems. Away from the projective cell, the growth
of KL polynomials along rays can behave quite differently; see also
\autoref{S:C2}. Nevertheless, the expected tensor power exponent
\(\tau=-|\Phi^+|/2\) still appears to persist.
\end{Remark}

\subsection{Growth of polynomials}\label{sec:pola2}
By \autoref{thm:poly-projective-ray}, all regular rays starting at the fundamental alcove have cubic growth. We now treat the cyclic wall family and the two neighboring projective families running in the same wall direction. Up to cyclic relabeling and reversal, these are the families
$x_j,y_j,z_j$
from \autoref{sec:cyclic}.

\begin{Theorem}\label{thm:a2-wall-polynomial-growth}
For the wall rays in affine type \(\widetilde A_2\), one has
\[
\nu^\delta(b_{x_j}), \nu^\delta(b_{y_j}), \nu^\delta(b_{z_j})\in\Theta(j^2).
\]
Regular rays are covered by
\autoref{thm:poly-projective-ray} and have cubic growth.
\end{Theorem}

\begin{proof}
We use the explicit formulas of 
\cite[Lemma~1.4 and Proposition~1.6]{LiPa}. Let \(N_w\) denote the positive
standard sum over the Bruhat interval below \(w\). Then
\[
\nu^\delta(N_w)=|\{u\in W\mid u\leq w\}|.
\]
For \(w=x_j,y_j,z_j\), these Bruhat intervals have quadratic size:
\[
\nu^\delta(N_w)\in\Theta(j^2).
\]
By \cite[Proposition~1.6]{LiPa}, each of
$b_{x_j},b_{y_j},b_{z_j}$
is a uniformly bounded finite positive sum of such \(N\)-terms, with all indices differing from
\(j\) by a bounded amount, and with the leading term \(N_{x_j}\), \(N_{y_j}\),
respectively \(N_{z_j}\), occurring. Hence each
\(\nu^\delta(b_{x_j})\), \(\nu^\delta(b_{y_j})\), and
\(\nu^\delta(b_{z_j})\) is a uniformly bounded finite positive sum of quantities in
\(\Theta(j^2)\). Therefore
\[
\nu^\delta(b_{x_j}), \nu^\delta(b_{y_j}), \nu^\delta(b_{z_j}) \in\Theta(j^2).
\]
The proof is complete.
\end{proof}

\subsection{Growth of tensor powers}

We now study the sequence $\nu_n$.

\subsubsection{The number of cyclic summands is negligible.}\label{sec:cyclic}
Recall from \autoref{sec:a2results}, for $j \geq 4$, the definitions of some projective cell elements;
\[
y_j:=x_{j-2}j,
\quad
z_j:=x_{j-3}j,
\]
where in each case the final letter is the indicated simple reflection written in $\{1,2,3\}$ modulo $3$.

\begin{Lemma}\label{lem:cyclic-mult}
The following multiplication rules hold in $\Hecke({\widetilde A_2})$;
\[
b_{x_j}b_{j+1}=
\begin{cases}
b_{x_{j+1}}+b_{y_j},&\text{if }j\text{ is odd},\\[0.4em]
b_{x_{j+1}}+b_{y_j}+b_{z_j},&\text{if }j\text{ is even}.
\end{cases}
\]
\end{Lemma}

\begin{proof}
Found in \cite{LiPa}.
\end{proof}

Let $\pi_{cyc}$ denote the projection onto the span of cyclic KL basis elements, and accordingly, let $\nu(\pi_{cyc}(b_w^n))$
denote the sum of coefficients of the cyclic summands occurring in $b_w^n$.

\begin{Lemma}\label{lem:cyc0}
If $j=3k$, then $b_{x_j}^n$ contains exactly one cyclic summand, namely $b_{x_{jn}}$, meaning
\[
\nu(\pi_{cyc}(b_{x_j}^n))=1.
\]
\end{Lemma}

\begin{proof}
We use induction on the number of $b_i$ we multiply by. Firstly, note that the product $b_{x_j}^n$ occurs inside the longer cyclic product
\[
b_{x_j}\prod_{i=3k+1}^{3kn} b_i,
\]
where the indices are read cyclically modulo $3$. 

Without loss of generality, suppose $j=3k$ is odd. By \autoref{lem:cyclic-mult}, $b_{x_j}b_{j+1} = b_{x_{j+1}} + b_{y_j}$, so only one cyclic summand has appeared, $b_{x_{j+1}}$.

Now, consider the multiplication $b_{x_j} \prod_{i=3k+1}^{m} b_i = \sum_w a_w b_w$. We know by standard KL theory that this cyclic summand is $b_{x_m}$. Suppose that this is the only cyclic summand, and every other $b_w$ in this sum is projective. Then, multiplication by $b_{m+1}$ gives us $b_{x_{m+1}}$ from the cyclic summand, plus a sum of projective elements (since the projective cell is a two-sided ideal). Thus, 
\[b_{x_j} \prod_{i=3k+1}^{m+1} b_i\]
also only has one cyclic summand. So, by induction, 
\[b_{x_j}\prod_{i=3k+1}^{3kn}b_i\]
contains exactly one cyclic summand, which is equal to $b_{x_{3kn}} = b_{x_{nj}}$.

Since $b_{x_j}^n$ appears as a summand of this larger product, it also contains at most one cyclic summand. Since we know it contains $b_{x_{nj}}$ by KL theory,it contains exactly this one cyclic summand. 
\end{proof}

\begin{Lemma}\label{lem:cyc1}
If $j=3k$, then $b_{x_{j+1}}^n$ contains exactly $2^{n-1}$ cyclic summands, all equal to $b_{x_{jn+1}}$. In particular,
\[
\nu(\pi_{cyc}(b_{x_{j+1}}^n))=2^{n-1}.
\]
\end{Lemma}

\begin{proof}
From \autoref{lem:cyclic-mult}, we see that $b_{x_{j+1}}^n = (b_{x_j}b_1 - b_{y_j})^n$. Then, the only cyclic element appearing in this expansion is $(b_{x_j}b_1)^n$, and everything else is projective. Then, we see;
\[
(b_{x_j}b_1)^n
= b_{x_j}(b_1b_{x_j})^{n-1}b_1
=2^{n-1}b_{x_j}^n b_1.
\]
Thus, $b_{x_{j+1}}^n$ contains one cyclic summand, $b_{x_{nj+1}}$, with multiplicity $2^{n-1}$.
\end{proof}

\begin{Lemma}\label{lem:cyc2}
If $j=3k$, then $b_{x_{j+2}}^n$ contains exactly one cyclic summand, namely $b_{x_{jn+2}}$. In particular,
\[
\nu(\pi_{cyc}(b_{x_{j+2}}^n))=1.
\]
\end{Lemma}

\begin{proof}
One argues exactly as above, now starting from
\[
(b_{x_j}b_{x_2})^n=b_{x_j}(b_{x_2}b_{x_j})^{n-1}b_{x_2}.
\]
After rewriting the middle factor, the same argument shows that at most one cyclic summand survives, and the cyclic term is $b_{x_{jn+2}}$. Hence this is the unique cyclic summand.
\end{proof}

\begin{Lemma}\label{cor:cyc-negl}
For cyclic elements in type $\widetilde A_2$, the cyclic part is asymptotically negligible:
\[
\frac{\nu(\pi_{cyc}(b_w^n))}{\nuop(b_w^n)}\longrightarrow 0.
\]
\end{Lemma}

\begin{proof}
Combine \autoref{lem:cyc0}, \autoref{lem:cyc1}, and
\autoref{lem:cyc2} with the general lower bound from
\autoref{thm:general-lower-upper}; for wall elements, the same lower bound
follows by \autoref{lem:lower-by-standard-sum}, using
\autoref{thm:a2-wall-polynomial-growth} in place of the projective-cell
polynomial bound. For $j\geq 3$ one has
\[
\nuop(b_{x_j}^n)\geq (j-1)^n,\quad n\gg 0,
\]
whereas the cyclic contribution is either $1$ or $2^{n-1}$. Since $j>2$, the ratio tends to zero.
\end{proof}

\begin{Lemma}\label{prop:a2reduction}
Let $w$ be an element in type $\widetilde A_2$, not contained in a parabolic subgroup, and let
$\pi_{proj}$
denote the projection onto the span of KL basis elements in the projective cell.
Then
\[
\nuop(b_w^n)\sim \nuop\bigl(\pi_{proj}(b_w^n)\bigr).
\]
In particular, the full asymptotic growth of $b_w^n$ is determined by its projective cell part.
\end{Lemma}

\begin{proof}
If $w$ lies in the projective cell, there is nothing to prove, since the span of the projective cell is a two-sided ideal (by the definition of cells).

The summands of $b_w^n$ are all contained within the projective, or cyclic cells. Let $\pi_{proj},\pi_{cyc}$ denote the projection onto these cells respectively.
Assume now that $w$ is cyclic on the wall. Decompose
\[
b_w^n=\pi_{cyc}(b_w^n)+\pi_{proj}(b_w^n).
\]
Since KL structure constants are nonnegative, both terms have nonnegative coefficients, and hence
\[
\nuop(b_w^n)
=
\nuop\bigl(\pi_{cyc}(b_w^n)\bigr)
+
\nuop\bigl(\pi_{proj}(b_w^n)\bigr).
\]
By \autoref{cor:cyc-negl}, we have that \[
\nuop\bigl(\pi_{cyc}(b_w^n)\bigr)
\in
o\bigl(\nuop(b_w^n)\bigr).
\]
Then, since
\[
\nuop(b_w^n)
=
\nuop\bigl(\pi_{proj}(b_w^n)\bigr)
+
\nuop\bigl(\pi_{cyc}(b_w^n)\bigr),
\]
it follows that
\[
\nuop(b_w^n)\sim \nuop\bigl(\pi_{proj}(b_w^n)\bigr).
\]
This proves the claim.
\end{proof}

\begin{Remark}\label{rem:a2reduction-pf}
In the language of, for example, \cite{LaTuVa-growth-pfdim,LaTuVa-growth-pfdim-inf}, the content of \autoref{prop:a2reduction} is really a Perron--Frobenius statement.

Indeed, multiplication by a fixed cyclic on-the-wall element $b_w$ defines a nonnegative operator on the span of the three relevant cells: the trivial cell, the cyclic wall cell, and the projective cell. With respect to this decomposition, the operator is block upper triangular, because the span of the projective cell is a two-sided ideal and hence a final class.

The diagonal blocks corresponding to the trivial cell and the cyclic wall cell have spectral radius at most $2$; in the cyclic case this is exactly what underlies \autoref{cor:cyc-negl}. By contrast, the block on the projective cell has spectral radius strictly larger than $2$, by \autoref{thm:general-lower-upper}. Therefore the projective block is the unique final class and the unique basic class.

Standard Perron--Frobenius theory for nonnegative matrices then implies that the contribution of all other classes is $o(\rho^n)$ relative to the projective part, where $\rho>2$ is the spectral radius of the projective block. In particular,
\[
\nuop(b_w^n)\sim \nuop\bigl(\pi_{proj}(b_w^n)\bigr).
\]

So the proof above is simply the concrete coefficient-sum version of this general Perron--Frobenius mechanism.
\end{Remark}

\subsubsection{Type affine \texorpdfstring{$A_2$}{A2}: One-sided and Two-sided Longest Elements}\label{sec:a2results}

We now record the ``uniform bound'' computation beyond the wall. We briefly recall the key combinatorial input is the following criterion, from \cite{LiPa}.

We briefly recall the definition of a braid distance below, and an important lemma, as in \cite{LiPa}.

\begin{Definition}\label{def:braid-triplets} If $w = r_1r_2...r_n$ is a (not necessarily reduced) expression for $w$, where $r_i$'s are simple reflections, we say that there is a braid triplet in position $i$ $(1 < i < n)$ if $r_{i-1} = r_{i+1}$, and $r_i \neq r_{i-1}$. 

We define the distance between a braid triplet in position $i$ and position $j > i$ to be the number $i-j-1$.
\end{Definition}

\begin{Lemma}\label{lem:lipa-criterion}
Assume a word has no adjacent equal simple reflections.
Then it is reduced if and only if the distance between any two braid triplets is odd.
\end{Lemma}

\begin{proof}
See \cite[Lemma~1.1]{LiPa}.
\end{proof}

Now, by the definition of $\theta(m,n)$, multiplication on the right by the simple reflection $2m-2n-1$, written modulo 3, induces an even distance between braid relations, thus reducing the word by \autoref{lem:lipa-criterion}. Of course, multiplication on the right by the simple reflection $2m-2n+1$ induces a quadratic relation, and so also reduces the word. Thus, we have the easy multiplication rules
\begin{gather}\label{eq:theta4}
\begin{aligned}
b_{\theta(m,n)}b_{{2m-2n+1}}&=2b_{\theta(m,n)},\\
b_{\theta(m,n)}b_{{2m-2n-1}}&=2b_{\theta(m,n)}.
\end{aligned}
\end{gather}
In other words, multiplying on the right by a simple reflection which induces either a quadratic relation
or a braid relation simply doubles the element.

We also use the following beyond-the-wall multiplication rules.

\begin{Lemma}
We have
\begin{gather}\label{eq:theta5}
\begin{aligned}
b_{\theta(m,n)}b_{{2m-2n}}&=b_{\theta(m,n){(2m-2n})},\\
b_{3}b_{\theta(m,n)}&=b_{3\theta(m,n)},\\
b_{3}b_{\theta(m,n)}b_{({2m-2n})}&=b_{3\theta(m,n)(2m-2n)},\\
b_{\theta(m,n)}b_{(2m-2n)}b_{(2m-2n-1)}
&=b_{\theta(m,n+1)}+b_{\theta(m,n)}+b_{\theta(m+1,n-1)}+b_{\theta(m-1,n)},
\end{aligned}
\end{gather}
where terms $\theta(x,y)$ with $x$, or $y$ negative are set to zero.
\end{Lemma}

\begin{proof}
See \cite[Proposition~1.8]{LiPa}.
\end{proof}

Recall also that in finite type $A_2$,
\[
b_{w_0}=b_1b_2b_1-b_1=b_2b_1b_2-b_2.
\]

\begin{Lemma}\label{prop:uniform6}
For beyond-the-wall elements $x$ in type $\wt A_2$ one has the uniform estimate
\[
\nuop\bigl(b_xb_{w_0}\bigr)\le 12.
\]
\end{Lemma}

\begin{proof}
There are three cases, depending on which simple reflection $\theta(m,n)$ ends in.
In the first case one expands
\[
b_{\theta(m,n)}b_{w_0}=b_{\theta(m,n)}(b_2b_1b_2-b_2),
\]
uses \autoref{eq:theta4} to multiply by $b_2b_1$, and then applies \autoref{eq:theta5}
to the remaining right multiplication by $b_2$.
This produces at most six summands.

In the second case one chooses the simple reflections so that successive right multiplications by
$b_1$, $b_2$, and $b_1$ all act by doubling, which again yields at most six summands after subtracting the final $b_1$ term.

The third case is similar: one first multiplies by $b_1$, then by $b_2$, and then analyses the final multiplication by $b_1$ using the same local rules.
The same bound follows.

For elements beyond-the-wall not equal to $\theta(m,n)$ i.e. those of the form $\theta(m,n)(2m-2n)$, working through each case we tend to get double the amount of summands, giving us 12. We postpone the technical computation deriving this to \autoref{S:Details}.
\end{proof}

Recall $\beta=\nu^\delta(b_w)$.

\begin{Theorem}\label{thm:a2-longest}
Let $w \in \widetilde A_2$.
\begin{enumerate}
\item
If $w \in \widetilde A_2$ is two-sided longest with respect to a parabolic subgroup $A_2 < \widetilde{A_2}$, then
\[
\nuop(b_w^n)\sim C_w\cdot n^{-3/2}\cdot\beta^n
\]
for some constant $C_w>0$.

\item
If $w \in \widetilde A_2$ is one-sided longest, then there exist constants $c_1(w), c_2(w)>0$ such that
\[
c_1(w)\cdot n^{-3/2}\cdot\beta^n
\;\le\;
\nuop(b_w^n)
\;\le\;
c_2(w)\cdot n^{-3/2}\cdot\beta^n
\]
for all $n\gg 1$.
Equivalently,
\[
\nuop(b_w^n)\in\Theta(n^{-3/2}\cdot \beta^n).
\]
\end{enumerate}
\end{Theorem}

\begin{proof}
Part (a) is \autoref{prop:satake-two-sided}, since $|\Phi^+|=3$ in type $A_2$.

Now suppose that $w$ is one-sided longest on the right, so that $b_we=b_w$. Note that $w$ being one-sided longest implies;
\[\nu((eb_we)^n) = \nu((eb_w)^n) = \nu(eb_w^n).\]
Since $\nu(eb_w^n) \sim C_w \cdot n^{\frac{-3}{2}}\cdot \beta^n$ by \autoref{prop:satake-two-sided}, this implies that there exist constants $A_w, B_w$ such that \[A_w \cdot n^{\frac{-3}{2}}\cdot\beta^n \leq \nu(eb_w^n) \leq B_w \cdot n^{\frac{-3}{2}}\cdot\beta^n.\]
Utilising \autoref{prop:uniform6}, we note that $\frac{1}{2}\nu(eb_w^n) \leq \nu(b_w^n) \leq 6\nu(eb_w^n)$. Thus, \[\frac{1}{2}A_w \cdot n^{\frac{-3}{2}}\cdot\beta^n \leq \nu(b_w^n) \leq 6 B_w \cdot n^{\frac{-3}{2}}\cdot\beta^n.\]
The theorem is proved.
\end{proof}

\begin{Remark}
The argument for the one-sided longest case uses only one input beyond Satake: a uniform bound on $\nuop(b_xb_{w_0})$ for $x$ in that big cell.
Consequently, in any affine Weyl group, if one knows that
\[
\nuop(b_xb_{w_0})\le C
\]
for all $x$ in the big cell, then the same proof gives the analogue of the one-sided part of \autoref{thm:a2-longest}, with exponent $n^{-|\Phi^+|/2}$.
\end{Remark}

\begin{Remark}
At this point, in type $\wt A_2$ the expected $n^{-|\Phi^+|/2}=n^{-3/2}$ growth is established for all longest elements:
two-sided longest elements are covered directly by Satake, while one-sided longest elements are controlled by the one-sided Satake sandwich together with the reduction to the projective cell and \autoref{prop:uniform6}. This makes roughly $1/3$ of all elements.

What remains outside this picture are the elements that are not one-sided longest from either side.
Among the cyclic on-the-wall families, this unresolved case is exactly one of the three residue classes modulo $3$, so asymptotically it accounts for one third of the wall elements, up to the usual lower-order boundary effects.
\end{Remark}

\section{The notorious cell in affine type C2}\label{S:C2}

Let us look at affine type \(C_2\), and in particular at its notorious wall. Let the affine Coxeter generators be \(1,2,3\), so that
\[
\widetilde C_2
=
\langle 1,2,3
\mid
1212=2121,\ 1313=3131,\ 23=32,\ 1^2=2^2=3^2=e
\rangle.
\]
The picture to keep in mind is \autoref{fig:b2-alcoves}.

\begin{figure}[ht]
\centering
\includegraphics[height=6cm]{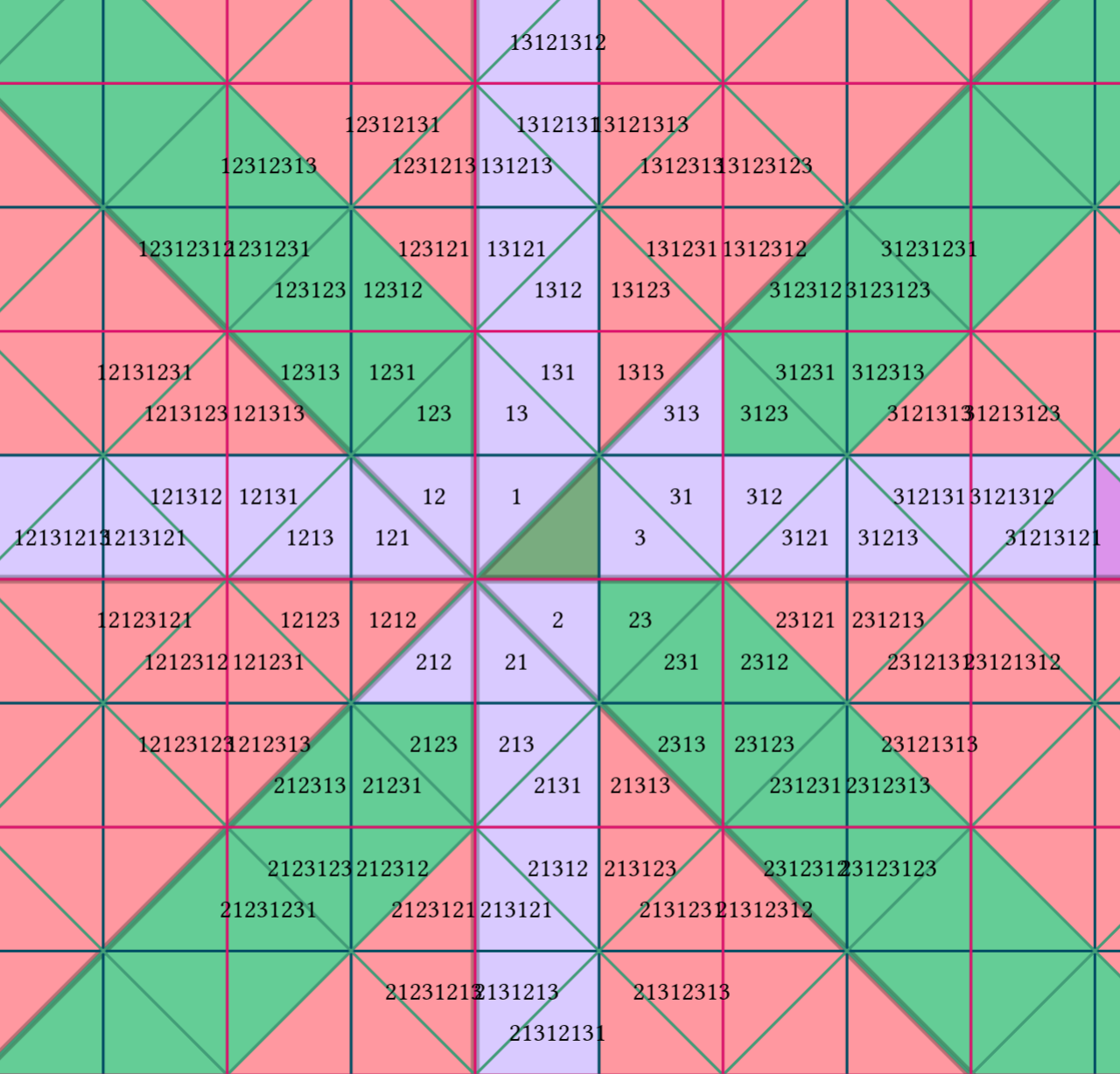}
\includegraphics[height=6cm]{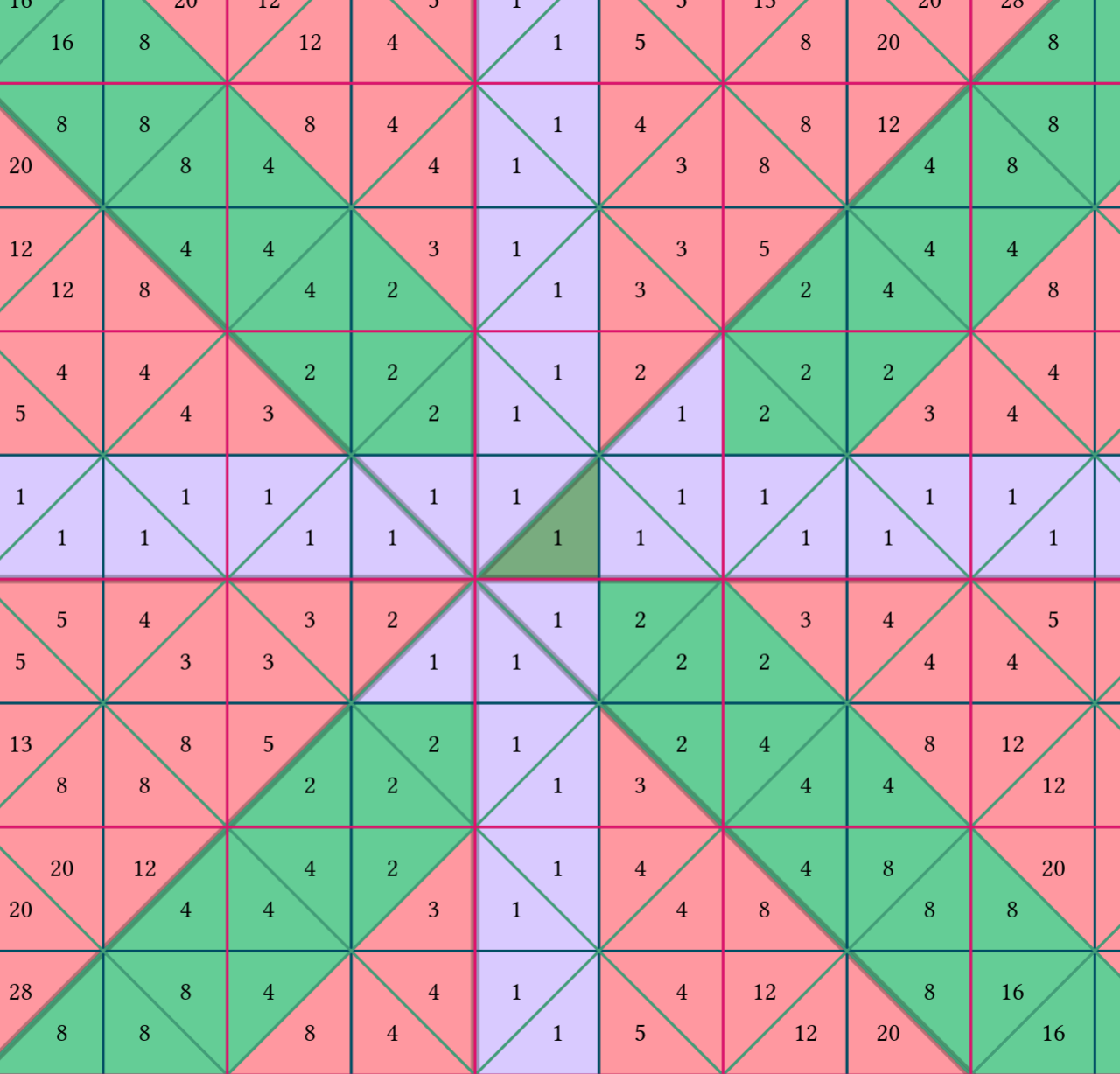}
\includegraphics[height=6cm]{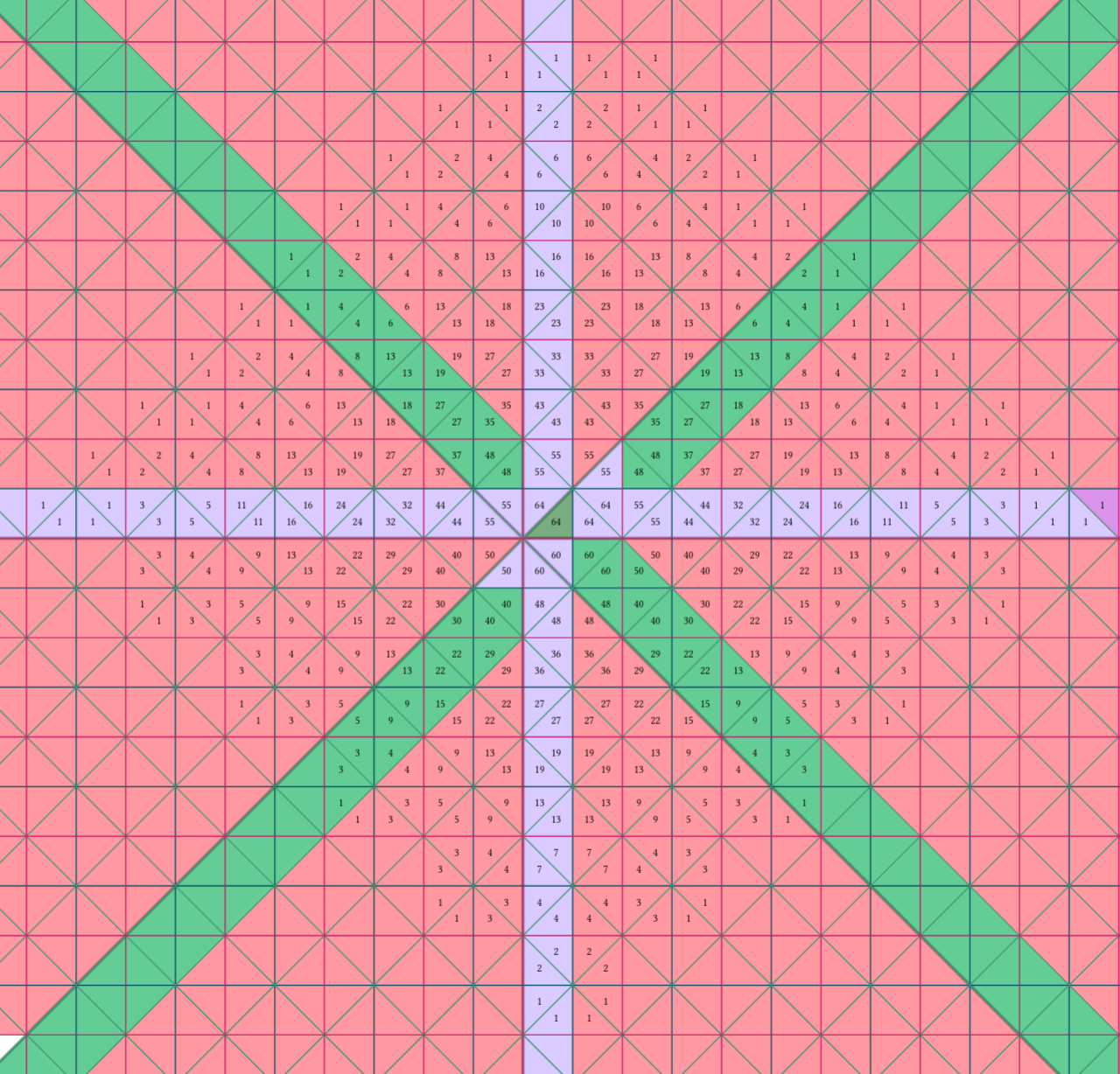}
\includegraphics[height=6cm]{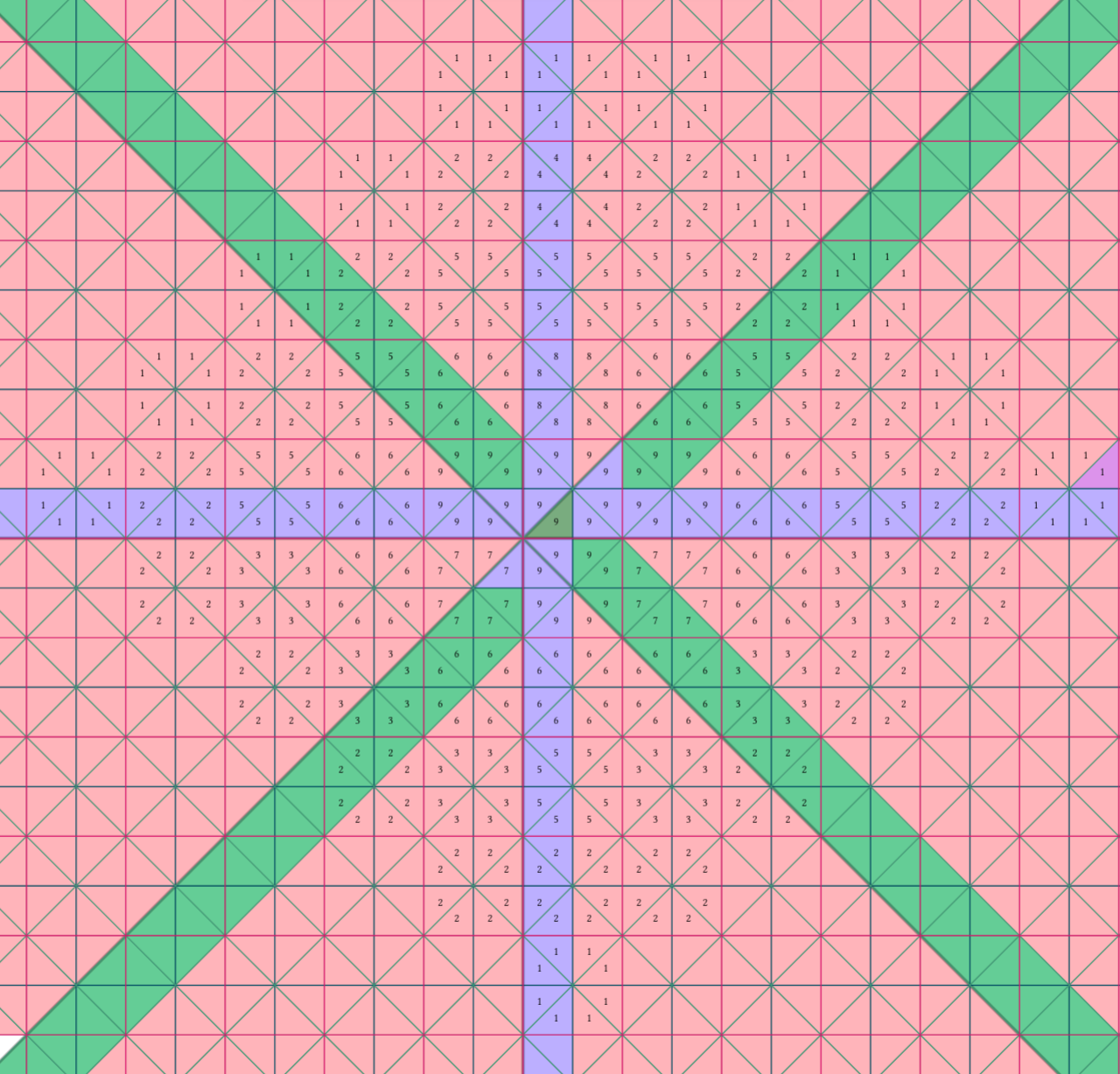}
\caption{The affine \(\widetilde C_2\) alcove picture, with hyperplane colors
\(1=\) green, \(2=\) red, and \(3=\) blue. In addition to the alcoves, the
picture shows the two-sided KL cells: the shaded central triangle is the trivial
cell, i.e. the fundamental alcove; the purple belt is the reduced expression cell, aka the wall cell; the
red cones form the projective cell; and there is one additional green cell.
Each alcove is labeled by information about the corresponding element \(w\in W\):
northwest gives a reduced expression, northeast the number of reduced expressions,
and southwest and southeast the KL polynomials for the shaded alcove at \(v=1\).}
\label{fig:b2-alcoves}
\end{figure}

\begin{Remark}
The calculations in this section were obtained with the code in \cite{COT}.
Some of them can be justified using the formulas of
\cite{BaBiPl-KL-B2}: the formulas for the neighboring rays in the
projective cell are unconditional, while the formulas along the thin wall
are conditional on \cite[Conjecture~6.1]{BaBiPl-KL-B2}. The tensor power
calculation at the end of the section remains experimental.
\end{Remark}

The elements along the wall Ignoring 313) are
\[
3,\ 31,\ 312,\ 3121,\ 31213,\dots,(3121)^m,
(3121)^m3,
(3121)^m31,
(3121)^m312,
(3121)^{m+1},
\dots.
\]
We write \(w_i\) for the wall element of length \(i\), with \(w_0=id\).
As one can see in \autoref{fig:b2-alcoves}, their KL polynomials are already
very large when compared with those of nearby alcoves. Indeed, one gets
\[
\big(\nu^\delta(b_{w_i})\big)_{i\in\Z_{\geq 0}}
=
(1,2,4,8,12,24,36,72,108,216,252,488,532,1016,988,1832,1716,3144,\dots).
\]

Let us compare this with \cite{BaBiPl-KL-B2}. Their generators
\(s_0,s_1,s_2\) correspond to our generators \(3,2,1\), respectively.
Let \(d_j\) be the length-\(j\) prefix of
\[
s_2s_1s_2s_0s_2s_1s_2s_0\cdots,
\qquad
\overline d_j=s_0d_j.
\]
Then
\[
w_i=\overline d_{i-1}
\qquad (i\geq2).
\]
We write \(\vartheta(a,b)\) for the element denoted by \(\theta(a,b)\)
in \cite{BaBiPl-KL-B2}, to avoid confusion with the notation used in
\autoref{sec:sl3}.

The computed values fall into a period-four pattern. More precisely,
assuming \cite[Conjecture~6.1]{BaBiPl-KL-B2}, for \(m\geq1\) one has
\[
\begin{aligned}
\nu^\delta(b_{w_{4m+1}})
&=
4(3m^4+m^2+2),
\\
\nu^\delta(b_{w_{4m+2}})
&=
4(2m^2+1)(m^2+m+1),\\
\nu^\delta(b_{w_{4m+3}})
&=
4(3m^4+6m^3+5m^2+2m+2),\\
\nu^\delta(b_{w_{4m+4}})
&=
4(m^2+m+1)(2m^2+4m+3).
\end{aligned}
\]

Here is how these formulas follow from the conjecture. At \(v=1\), summing
the coefficients in the standard basis is the augmentation map
\[
\varepsilon\colon\Hecke=\R[W]\longrightarrow\R,
\qquad
\varepsilon\big(\sum_{x\in W}a_x\delta_x\big)=\sum_{x\in W}a_x,
\]
which is an algebra homomorphism. Thus
\(\nu^\delta(h)=\varepsilon(h)\). By
\cite[Theorem~1.1 and Corollary~2.5]{BaBiPl-KL-B2},
\[
A_0(j):=\nu^\delta(b_{\vartheta(0,j)})
=
\tfrac{4}{3}(j+1)(j+2)(2j+3)
\]
and
\[
A_1(j):=\nu^\delta(b_{\vartheta(1,j)})
=
\tfrac{16}{3}(j+1)(j+2)(j+3).
\]
Applying \(\varepsilon\) to the four identities in
\cite[Conjecture~6.1(2)]{BaBiPl-KL-B2} gives recurrences for
\[
D_j:=\nu^\delta(b_{d_j}).
\]
For example,
\[
2D_{4m+1}
=
D_{4m+2}
+
\sum_{j=0}^{m-2}A_0(j)
+
3\sum_{j=0}^{m-3}A_0(j).
\]
Together with the other three identities and the initial values
\(D_3=6\) and \(D_4=12\), these determine all \(D_j\).
Applying \(\varepsilon\) to
\cite[Conjecture~6.1(3)]{BaBiPl-KL-B2} then determines
\(\nu^\delta(b_{\overline d_j})\). Substituting the formulas for
\(A_0(j)\) and \(A_1(j)\) and summing gives exactly the four formulas above.

In fact, the same calculation gives quartic quasi-polynomial formulas for
the elements \(d_j\). The other two arms of the thin cell are obtained from
\(d_j\) and \(\overline d_j\) by the diagram automorphism and have the same
values. Thus, conditional on \cite[Conjecture~6.1]{BaBiPl-KL-B2}, every
infinite arm of the thin cell has quartic quasi-polynomial growth.

It is instructive to compare this with two neighboring rays in the projective cell, one above and one below the wall in \autoref{fig:b2-alcoves}. Above the wall,
start at
$3121313$
and then extend on the right by \(2,1,3,1\), repeatedly. Thus we obtain
\[
3121313,\quad
31213132,\quad
312131321,\quad
3121313213,\quad
31213132131,\quad\dots.
\]
Let \(u_i\) denote the resulting element of length \(i\). Then
\[
\big(\nu^\delta(b_{u_i})\big)_{i\in\Z_{\geq 7}}
=
(32,64,96,128,160,320,480,640,448,896,1344,1792,\dots),
\]
and, for \(m\geq 0\) and \(r\in\{0,1,2,3\}\),
\[
\nu^\delta(b_{u_{7+4m+r}})
=
\tfrac{16(r+1)}{3}\cdot(m+1)(m+2)(2m+3).
\]
Below the wall, start at $23121$
and then extend on the right by \(3,1,2,1\), repeatedly. This gives
\[
23121,\quad
231213,\quad
2312131,\quad
23121312,\quad
231213121,\quad\dots.
\]
Let \(v_i\) denote the resulting element of length \(i\). Then
\[
\big(\nu^\delta(b_{v_i})\big)_{i\in\Z_{\geq 5}}
=
(16,32,48,64,80,160,240,320,224,448,672,896,\dots),
\]
and, for \(m\geq 0\) and \(r\in\{0,1,2,3\}\),
\[
\nu^\delta(b_{v_{5+4m+r}})
=
\tfrac{8(r+1)}{3}\cdot(m+1)(m+2)(2m+3).
\]

These two formulas follow unconditionally from
\cite[Theorem~1.1]{BaBiPl-KL-B2}. Indeed, set
\[
y_0=1,\qquad
y_1=s_0,\qquad
y_2=s_0s_2,\qquad
y_3=s_0s_2s_1,
\]
and let \(y'_r\) be obtained by interchanging \(s_0\) and \(s_1\). In the
notation of \cite{BaBiPl-KL-B2}, one has
\[
v_{5+4m+r}=s_0\vartheta(0,m)y_r
\]
and
\[
u_{7+4m+r}
=
s_0s_2s_1\vartheta'(0,m)y'_r.
\]
Moreover,
\[
\begin{aligned}
\nu^\delta(b_{\vartheta(0,m)})
&=
\sum_{\substack{0\leq j\leq m\\j\equiv m\!\!\!\pmod 2}}
8(2j^2+2j+1)\\
&=
\tfrac{4}{3}(m+1)(m+2)(2m+3).
\end{aligned}
\]
After applying the augmentation map, the left factors contribute \(2\)
and \(4\), respectively, while the right factor contributes \(r+1\).
This gives the two displayed formulas. In particular, both neighboring
rays in the projective cell grow cubically, whereas, conditional on
\cite[Conjecture~6.1]{BaBiPl-KL-B2}, the thin wall grows quartically.

\begin{Remark}
Assuming \cite[Conjecture~6.1]{BaBiPl-KL-B2}, the analog of
\autoref{lem:singular rays} is therefore not true outside of the projective
cell in general.
\end{Remark}

Finally, let us return from polynomial growth along rays to tensor powers.
The element \(312\) is a natural affine \(\widetilde C_2\) analog of the
cyclic wall element \(123\) in type \(\widetilde A_2\). Notice that
\(312=\overline d_2\) in the notation above. However, the formulas of
\cite{BaBiPl-KL-B2} describe the standard-basis expansions of individual
KL basis elements and do not give a closed recurrence for multiplication
by \(b_{\overline d_2}\). Thus they do not directly settle the following
tensor-power problem.

Computations give
\begin{gather*}
\bigl(\nuop(b_{312}^n)\bigr)_{n\geq 1}
=
(1,4,18,94,537,3204,19678,123706,793667,5185460,
34435639,\\
232019728,1583495458,
10930233768,76204662146,
\dots).
\end{gather*}
The normalized values
$\nuop(b_{312}^n)/(n^{-2}8^n)$
appear to stabilize rapidly, as shown in \autoref{fig:c2-312-normalized}.
This is good numerical evidence for
\[
\nuop(b_{312}^n)\sim C\cdot n^{-2}\cdot 8^n
\]
for some constant \(C>0\).

\begin{figure}[ht]
\centering
\includegraphics[height=6cm]{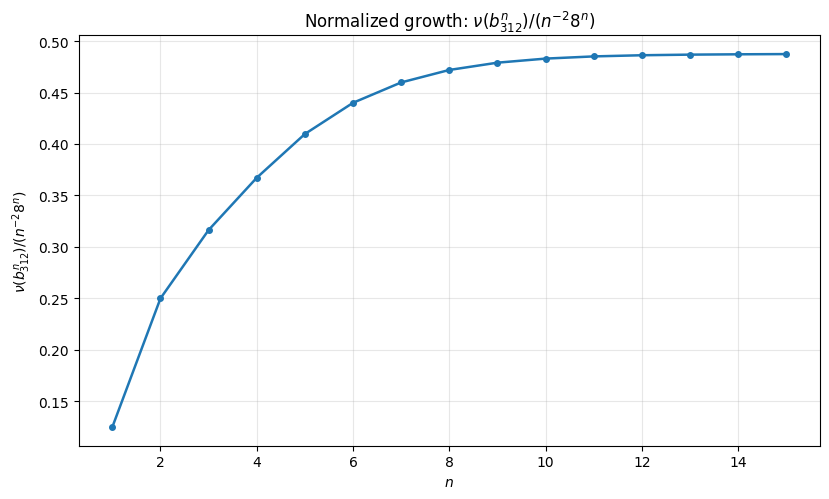}
\caption{The (normalized) sequence $\nu_n$ for $w=312$.}
\label{fig:c2-312-normalized}
\end{figure}

\begin{Remark}
The exponent is the expected one, but we do not attempt to prove this here.
\end{Remark}

\section{Details for Proof of \autoref{prop:uniform6}}\label{S:Details}
\setcounter{subsection}{1}

We now give the bookkeeping details behind \autoref{prop:uniform6}. Below, at one point in the proof of \autoref{lem:theta-w0-bound}, we will return to the $\Z[v,v^{-1}]$-algebra $\Hecke_v(W)$. Recall that this is the unital associative algebra generated by symbols $\{\delta_x : x\in S\}$, subject to the quadratic relation;
\begin{align*}
\delta_x^2 &= (v^{-1}-v)\delta_x + 1,
\end{align*}
and the braid relation;
\begin{align*}
\delta_x\delta_y\delta_x... &= \delta_y\delta_x\delta_y...,
\end{align*}
where there are $m_{x,y}$ terms on each side (here, $m_{x,y}$ is the smallest positive integer such that $(xy)^{m_{x,y}}=id$ in $W$).

Throughout
this section, contrary to \autoref{sec:sl3}, all labels of simple reflections are read modulo \(3\), and we use $s_j$ instead of $j$ for the simple reflections. Furthermore, we use $s,t$ as the generators of $A_2$, and $u$ as the affine reflection. We fix $s = s_1$, $t = s_2$, and $u = s_3$, so that $w_0=sts$. We use the multiplication rules from \cite{LiPa}, recalled in
\autoref{eq:theta4} and \autoref{eq:theta5}. 

\begin{Remark}
Note in the following proofs, if we were to swap the roles of $s$ and $t$ in any given expression, we would get the same results in each case by working with the respective other expression for $b_{w_0}$. Thus, the below cases actually cover all possible $\theta$'s.
\end{Remark}

\begin{Lemma}\label{lem:theta-w0-bound}
In \(\Hecke(\wt A_2)\), one has
\[
\nuop\bigl(b_{\theta(m,n)}b_{w_0}\bigr)\leq 6
\]
for all beyond-the-wall words \(\theta(m,n)\).
\end{Lemma}

\begin{proof}
We work in $\Hecke(\widetilde{A_2})$. Recall that
\[
\theta(m,n)
=
123\cdots(2m+1)(2m+2)(2m+1)2m\cdots(2m-2n+1),
\]
and that, in finite type \(A_2\),
\[
b_{w_0}=b_tb_sb_t-b_t=b_sb_tb_s-b_s.
\]
We consider three cases, depending on the final simple reflection of
\(\theta(m,n)\).

\smallskip

\noindent\emph{Case 1.}
Suppose \(\theta(m,n)\) ends in \(u\), i.e. 
\[
u=s_{2m-2n+1} = s_3,
\]
so $2m-2n+1$ is a multiple of three. 
Hence,
\[
s= s_1 = s_{2m-2n-1},
\quad
t=s_2=s_{2m-2n}.
\]
Then, $\theta(m,n) = ...tsu$. Now,
\[
b_{\theta(m,n)}b_{w_0}
=
b_{\theta(m,n)}(b_tb_sb_t-b_t).
\]
By \autoref{eq:theta5},
\[
b_{\theta(m,n)}b_tb_s
=
b_{\theta(m,n+1)}
+
b_{\theta(m,n)}
+
b_{\theta(m+1,n-1)}
+
b_{\theta(m-1,n)}.
\]
We now multiply this expression on the right by \(b_t=b_{s_{2m-2n}}\). The three
terms different from \(\theta(m,n)\) all end in \(s\), because
\[
\begin{aligned}
2m-2(n+1)+1 &\equiv 2m-2n-1,\\
2(m+1)-2(n-1)+1 &\equiv 2m-2n-1,\\
2(m-1)-2n+1 &\equiv 2m-2n-1.
\end{aligned}
\]
Moreover, in each case the final \(s\) is preceded by \(t\). By definition, any $\theta(x,y)$ has exactly one braid triplet in it. Multiplication on the RHS by $t$ generates a second braid triplet, at position $2m + 2n + 3$. The induced braid distance is then $2m + 2n + 3 - (2m + 2) - 1$, which is even. Thus, by \autoref{lem:lipa-criterion}, $\theta(x,y)t$ is non-reduced, so $\theta(x,y)$ is a right-longest representative of $W/\{1,t\}$. Hence
\[
b_{\theta(x,y)}b_t=2b_{\theta(x,y)}
\]
for
\[
(x,y)\in\{(m,n+1),(m+1,n-1),(m-1,n)\}.
\]
Therefore
\[
\begin{aligned}
b_{\theta(m,n)}(b_tb_sb_t-b_t)
&=
\bigl(
b_{\theta(m,n+1)}
+
b_{\theta(m+1,n-1)}
+
b_{\theta(m-1,n)}
+
b_{\theta(m,n)}
\bigr)b_t
-
b_{\theta(m,n)}b_t
\\
&=
2\bigl(
b_{\theta(m,n+1)}
+
b_{\theta(m+1,n-1)}
+
b_{\theta(m-1,n)}
\bigr).
\end{aligned}
\]
Thus
\[
\nuop\bigl(b_{\theta(m,n)}b_{w_0}\bigr)\leq 6.
\]

\smallskip

\noindent\emph{Case 2.}
Suppose now that $2m-2n$ is a multiple of three, so that $u = s_{2m-2n}$. Then, $s = s_1 = s_{2m-2n+1}$, and $t=s_2=s_{2m-2n-1}$. Then, $\theta(m,n) = ...uts$. Following the multiplication rules;
\[
b_{\theta(m,n)}b_s=2b_{\theta(m,n)}.
\]
Also, multiplying \(\theta(m,n)\) on the right by \(t\) creates a braid relation
of the form \(tst\). The corresponding braid distance is
\[
2m+2n+3-(2m+2)-1,
\]
which is even. Hence \(\theta(m,n)t<\theta(m,n)\), and therefore
\[
b_{\theta(m,n)}b_t=2b_{\theta(m,n)}.
\]
It follows that
\[
\begin{aligned}
b_{\theta(m,n)}(b_sb_tb_s-b_s)
&=
8b_{\theta(m,n)}-2b_{\theta(m,n)}
\\
&=
6b_{\theta(m,n)}.
\end{aligned}
\]
Thus again
\[
\nuop\bigl(b_{\theta(m,n)}b_{w_0}\bigr)\leq 6.
\]

\smallskip

\noindent\emph{Case 3.}
Suppose now that $2m-2n-1$ is a multiple of three, so that $u = s_{2m-2n-1}$, $s = s_1 = s_{2m-2n}$, and $t = s_2 = s_{2m-2n+1}$. Then, $\theta(m,n) = ...sut$. 
First,
\[
b_{\theta(m,n)}b_t=2b_{\theta(m,n)}.
\]
Moreover, by \autoref{eq:theta5},
\[
b_{\theta(m,n)}b_s= b_{\theta(m,n)s}.
\]
So, we have \[b_{\theta(m,n)}b_tb_s = 2b_{\theta(m,n)s}.\] 
It remains to analyze multiplication by \(b_t\). This is the only bookkeeping
heavy case, and here we will change to working in $\Hecke_v(\widetilde{A_2})$ for a bit.

We claim that
\[
b_{\theta(m,n)}b_sb_t
=
b_{\theta(m,n)st}
+
b_{\theta(m,n)}
+
b_{\theta(m,n-1)}
+
b_{\theta(m-1,n)t}.
\]
The first summand occurs because \(\theta(m,n)st\) is reduced: the two braid
triplets occur at positions \(2m+2\) and \(2m+2n+4\), and their distance is thus $2m + 2n + 4 - (2m + 2) - 1$, which is
odd.

To find the remaining summands, write
\[
b_{\theta(m,n)}
=
\delta_{\theta(m,n)}
+
\sum_{z<\theta(m,n)}h_{z,\theta(m,n)}\delta_z,
\]
where $h_{x,\theta(m,n)}$ are polynomials in $v$. Then
\[
\begin{aligned}
b_{\theta(m,n)}b_s
&=
\delta_{\theta(m,n)s}
+
v\delta_{\theta(m,n)}
+
\sum_{z<\theta(m,n)}h_{z,\theta(m,n)}\delta_z b_s,
\\
b_{\theta(m,n)}b_sb_t
&=
\delta_{\theta(m,n)st}
+ 
v\delta_{\theta(m,n)s} 
+
v(v^{-1}-v)\delta_{\theta(m,n)}
+
v\delta_{\theta(m,n)t}\\&+
\sum_{z<\theta(m,n)}h_{z,\theta(m,n)}\delta_z b_sb_t 
\end{aligned}
\]
Thus \(b_{\theta(m,n)}\) is also a summand.

The remaining possibilities require us to observe when we can get a constant (no $v$-coefficient) in $h_{z,\theta(m,n)}\delta_zb_sb_t$. This leads us to the following conditions:
\[
v \text{ occurs in } h_{ys,\theta(m,n)}
\quad\text{and}\quad
ys<y>yt,
\]
or
\[
v^2 \text{ occurs in } h_{y,\theta(m,n)}
\quad\text{and}\quad
yst<ys<y.
\]

For the first condition, use the formula \(A_{m,n}\) from \cite{LiPa}:
\[
b_{\theta(m,n)}
=
\sum_{i=0}^{\min(m,n)}v^{2i}N_{\theta(m-i,n-i)},
\quad
N_x=\sum_{z\leq x}v^{\ell(x)-\ell(z)}\delta_z.
\]
For \(v\) to occur in \(h_{ys,\theta(m,n)}\), we need \(i=0\) and
\[
\ell(\theta(m,n))-\ell(ys)=1.
\]
By \cite[Corollary~1.3]{LiPa}, the relevant elements are obtained by removing one letter from $\theta(m,n)$ to obtain one of the following words:
\[
\begin{aligned}
&23\cdots(2m+1)(2m+2)(2m+1)\cdots(2m-2n+1),\\
&123\cdots 2m(2m+2)(2m+1)\cdots(2m-2n+1),\\
&123\cdots(2m+1)(2m+2)2m\cdots(2m-2n+1),\\
&123\cdots(2m+1)(2m+2)(2m+1)\cdots(2m-2n+2).
\end{aligned}
\]
Among these, only the second possibility contributes: it is the one for which
right multiplication by \(t\) is nonreduced. The corresponding summand is
\[
b_{\theta(m-1,n)t}.
\]

For the second condition, we look for \(v^2\) in \(h_{y,\theta(m,n)}\). Thus
either \(i=1\) and \(\ell(y)=\ell(\theta(m,n))\), or \(i=0\) and
\[
\ell(y)=\ell(\theta(m,n))-2.
\]
The first alternative does not contribute: the candidate \(\theta(m-1,n-1)\)
ends in \(t\), and right multiplication by \(s\) is reduced.

For the second alternative, one obtains \(y\) by deleting two simple reflections
from \(\theta(m,n)\). If neither of the final two letters is removed, then right
multiplication by \(s\) is reduced, so there is no contribution. Thus one of the
final two letters must be removed. A direct check using
\autoref{lem:lipa-criterion} shows that the only contributing case is obtained
by removing the final \(t\) and then the final \(u\). This gives
\[
y=\theta(m,n-1),
\]
and one has
\[
y>ys>yst.
\]
Thus \(b_{\theta(m,n-1)}\) is the final extra summand.

Combining the four summands gives
\[
b_{\theta(m,n)}b_sb_t
=
b_{\theta(m,n)st}
+
b_{\theta(m,n)}
+
b_{\theta(m,n-1)}
+
b_{\theta(m-1,n)t}.
\]
We now work in $\Hecke(\widetilde{A_2})$ again. Since \(b_{\theta(m,n)}b_t=2b_{\theta(m,n)}\), we obtain
\[
\begin{aligned}
b_{\theta(m,n)}b_tb_sb_t
&=
2b_{\theta(m,n)}b_sb_t
\\
&=
2\bigl(
b_{\theta(m,n)st}
+
b_{\theta(m,n)}
+
b_{\theta(m,n-1)}
+
b_{\theta(m-1,n)t}
\bigr).
\end{aligned}
\]
Subtracting \(b_{\theta(m,n)}b_t=2b_{\theta(m,n)}\), we get
\[
b_{\theta(m,n)}(b_tb_sb_t-b_t)
=
2b_{\theta(m,n)st}
+
2b_{\theta(m,n-1)}
+
2b_{\theta(m-1,n)t}.
\]
Hence
\[
\nuop\bigl(b_{\theta(m,n)}b_{w_0}\bigr)\leq 6.
\]
This completes the proof.
\end{proof}

The calculation in the third case also gives the following multiplication rule.

\begin{Remark}\label{rem:extra-beyond-wall-rule}
One has
\[
\begin{aligned}
b_{\theta(m,n)}b_{s_{2m-2n}}b_{s_{2m-2n+1}}
&=
b_{\theta(m,n)s_{2m-2n}s_{2m-2n+1}}
+
b_{\theta(m,n)}
\\
&\quad
+
b_{\theta(m,n-1)}
+
b_{\theta(m-1,n)s_{2m-2n+1}},
\end{aligned}
\]
as one can check.
\end{Remark}

\begin{Lemma}\label{cor:beyond-wall-w0-bound}
For every beyond-the-wall element \(w\) in type \(\wt A_2\), one has
\[
1\leq \nuop\bigl(b_wb_{w_0}\bigr)\leq 12.
\]
\end{Lemma}

\begin{proof}
By \autoref{lem:theta-w0-bound}, the estimate is already proved for
\(w=\theta(m,n)\). It remains to consider the other beyond-the-wall basis
elements, which are obtained by multiplying \(\theta(m,n)\) on the right by one
simple reflection.

\smallskip

\noindent\emph{Case 1.}
Suppose that \(\theta(m,n)\) ends in \(u\), thus $\theta(m,n) = ...stu$. From this,
\(\theta(m,n)s\) is reduced. Then
\[
\begin{aligned}
b_{\theta(m,n)s}b_{w_0}
&=
b_{\theta(m,n)}b_s(b_sb_tb_s-b_s)
\\
&=
2b_{\theta(m,n)}b_{w_0}.
\end{aligned}
\]
The same argument applies after interchanging \(t\) and \(s\).

\smallskip

\noindent\emph{Case 2.}
Suppose that $u = s_{2m-2n}$, and so $\theta(m,n) = ...tus$. Then
\[
\begin{aligned}
b_{\theta(m,n)t}b_{w_0}
&=
b_{\theta(m,n)}b_s(b_sb_tb_s-b_s)
\\
&=
2b_{\theta(m,n)}b_{w_0}.
\end{aligned}
\]

\smallskip

\noindent\emph{Case 3.}
Suppose that \(\theta(m,n)\) ends in \(st\). We need to estimate
\(b_{\theta(m,n)u}b_{w_0}\). Using \autoref{eq:theta5} and
\autoref{rem:extra-beyond-wall-rule}, we get
\[
\begin{aligned}
b_{\theta(m,n)u}b_{w_0}
&=
b_{\theta(m,n)}b_u(b_sb_tb_s-b_s)
\\
&=
b_{\theta(m,n)}b_{s_{2m-2n}}b_{s_{2m-2n+1}}(b_tb_s-1)
\\
&=
\bigl(
b_{\theta(m,n)s_{2m-2n}s_{2m-2n+1}}
+
b_{\theta(m,n)}
+
b_{\theta(m,n-1)}
+
b_{\theta(m-1,n)s_{2m-2n+1}}
\bigr)(b_tb_s-1).
\end{aligned}
\]
For each \(\theta(x,y)\neq\theta(m,n)\) appearing in this expression, the word
\(\theta(x,y)\) ends in \(us\). Hence
\[
t=s_{2x-2y},
\quad
s=s_{2x-2y+1},
\]
and \autoref{rem:extra-beyond-wall-rule} gives at most four summands after
multiplication by \(b_tb_s\). The term \(b_{\theta(m,n)}b_tb_s\) contributes at
most four summands as well. After subtracting the original four summands, we get
at most twelve summands in total.

The remaining cases are obtained from these by rotating the labels. Hence
\[
\nuop\bigl(b_wb_{w_0}\bigr)\leq 12
\]
for all beyond-the-wall elements \(w\). The lower bound is immediate.
\end{proof}

\end{document}